 \documentclass[12pt,oneside]{article}
\usepackage[T1]{fontenc}
\usepackage[round,colon,authoryear]{natbib} 
\usepackage{graphicx}
\usepackage{amsmath,amssymb}
\usepackage{amsfonts}
\usepackage{amssymb}
\usepackage{amsthm}
\usepackage{enumerate}
\usepackage{pdfsync}
\usepackage{times}
\usepackage{epic,graphpap,graphics,float,tabularx}
\usepackage{bm}
\usepackage[colorlinks,citecolor=blue]{hyperref}

\makeatletter
\newcommand{\zjel}[1]{\left({#1}\right)} 

\newcommand{\norm}[1]{\left\|{#1}\right\|} 

\def\isparamm{\@ifnextchar{\bgroup}}
\def\subscr#1{_{\{#1\}}}
\def\subnozjel#1{_{#1}}
\newcommand{\I}[1][]{\mathbf{1}%
  \@ifempty{#1}{\isparamm\subscr\subnozjel}{\subnozjel{#1}}}
\newcommand{\E}[1][]{\mathbb{E}\@ifempty{#1}{}{_{#1}}\isparamm\zjel\relax}
\renewcommand{\P}[1][]{\mathbb{P}\@ifempty{#1}{}{_{#1}}\isparamm\zjel\relax}
\newcommand{\PQ}[1][]{\mathbb{Q}\@ifempty{#1}{}{_{#1}}\isparamm\zjel\relax}

\newcommand{\sgn}{\@ifstar\sgnstar\sgnnostar}
\def\sgnnostar{\operatorname{sgn}}
\def\sgnstar{\overline{\operatorname{sgn}}}
\makeatother

\newcounter{enumr}

\numberwithin{equation}{section}

\newcommand{\p}{ \mathbb{P}}            
\newcommand{\R}{\mathbb{R}}            
\newcommand{\N}{\mathbb{N}}            

\newcommand{\de}{\mathrm{d}}  

\newcommand{\ra}{\rightarrow}           

\newcommand{\Om}{\Omega}

\bmdefine\xib{\mathbf{\xi}}
\bmdefine\mub{\mathbf{\mu}}
\bmdefine\nub{\mathbf{\nu}}
\bmdefine\etab{\mathbf{\eta}}
\bmdefine\deltab{\mathbf{\delta}}
\bmdefine\thetab{\mathbf{\vartheta}}
\bmdefine\alphab{\mathbf{\alpha}}
\bmdefine\betab{\mathbf{\beta}}
\bmdefine\taub{\mathbf{\tau}}
\bmdefine\sigmab{\mathbf{\sigma}}
\bmdefine\Sigmab{\mathbf{\Sigma}}
\bmdefine\gammab{\mathbf{\gamma}}
\bmdefine\Gammab{\mathbf{\Gamma}}

\theoremstyle{plain}
\newtheorem{thm}{Theorem}[section]

\newtheorem{prop}{Proposition}[section]

\theoremstyle{definition}

\theoremstyle{remark}
\newtheorem{rem}{Remark}[section]

\title{     Two 
Brownian Particles  with Rank-Based Characteristics    and  Skew-Elastic Collisions
} 

\vspace{2mm}

\author{
 \textsc{E. Robert Fernholz}
            \thanks{$\,$\textsc{Intech} Investment Management LLC, One Palmer Square, Suite 441,  Princeton, NJ 08542 (E-mail: {\it bob@enhanced.com}).}
             \and
 \textsc{Tomoyuki Ichiba}
  \thanks{$\,$Department of Statistics and Applied Probability, South Hall,
    University of California, Santa Barbara, CA 93106 (E-mail:  {\it
      ichiba@pstat.ucsb.edu}).} 
  \and
           \textsc{Ioannis Karatzas}
            \thanks{$\,$\textsc{Intech} Investment Management LLC, One Palmer Square, Suite 441,  Princeton, NJ 08542 (E-mail: {\it ik@enhanced.com}) and Department of Mathematics, Columbia University, MailCode 4438, New York, NY 10027 (E-mail: {\it ik@math.columbia.edu}). Research  supported in part by            National Science Foundation Grant DMS-09-05754.}
                       }  

\date{August 19, 2012} 

\begin{document}

\maketitle

\begin{abstract}
  \medskip 
\small  

\noindent   
We construct a two-dimensional diffusion process with rank-dependent local drift and dispersion co\"efficients, and with a full range of  patterns of behavior upon  collision that range from totally frictionless interaction, to elastic collision, to perfect reflection of one particle on the other. These interactions are governed by the left- and right-local times at the origin for the distance between the two particles. We realize this diffusion in terms of   appropriate, apparently novel systems of stochastic differential equations involving local times,   which we show are well posed. Questions of pathwise uniqueness and strength are also discussed for these systems.

The analysis depends crucially on properties of a skew Brownian motion with two-valued drift of the bang-bang   type, which we also study in some detail. These properties allow us to compute the transition probabilities of the original planar diffusion, and to study its behavior under time reversal. 
    \end{abstract}


 \input amssym.def
\input amssym

\medskip



\section{Introduction}
\label{sec1}

\noindent
We   construct a planar diffusion $\,   (X_1 (\cdot), X_2 (\cdot))\,$ according to the following recipe: each of its component particles  $\,  X_1 (\cdot)\,$ and $\, X_2 (\cdot) \,$    behaves locally like Brownian motion. The characteristics of these random motions are  assigned not by name, but by rank: the leader is assigned drift $\, -h \le 0\,$ and dispersion $\, \rho  \ge 0\,$, whereas the laggard is assigned drift $\, g \ge 0\,$ and dispersion $\, \sigma  \ge  0\,$. One of the dispersions is allowed to vanish, but not both; similarly for the drifts. In the interest of concreteness and simplicity, we shall  set 
\begin{equation}
\label{1.1}
\lambda \,:=\, g + h\,>\,0  \,,  \qquad 
\rho^2 + \sigma^2 \,=\,1  \,.
\end{equation} 
A bit more precisely, we shall   construct a complete  probability space $\, ( \Om,   \mathfrak{F}, \p)\,$ endowed with a filtration $\, \mathbf{F} = \{ \mathfrak{F} (t) \}_{0 \le t < \infty }\,$ that satisfies the ``usual conditions''  of right continuity and of augmentation by $\p-$negligible sets, and on it two pairs $\, (B_1 (\cdot), B_2 (\cdot))\,$ and $\, (X_1 (\cdot), X_2 (\cdot))\,$ of continuous, $\, \mathbf{F}-$adapted processes, such that  $\, (B_1 (\cdot), B_2 (\cdot))\,$ is planar Brownian motion  and $  (X_1 (\cdot),X_2 (\cdot))\,$ a continuous planar semimartingale that starts at some given site $  (X_1 (0), X_2 (0))= (x_1, x_2) \in \R^2\,$ on the plane   and  satisfies   the dynamics
$$
\mathrm{d} X_1 (t)\,=\,  \Big( g \, \mathbf{ 1}_{ \{ X_1 (t) \le  X_2 (t) \} } - h \, \mathbf{ 1}_{ \{ X_1 (t) >  X_2 (t) \} } \Big) \, \mathrm{d} t \,+ \,  \Big( \rho\, \mathbf{ 1}_{ \{ X_1 (t) >  X_2 (t) \} }  +\, \sigma\,  \mathbf{ 1}_{ \{ X_1 (t) \le   X_2 (t) \} } \Big)\,      \mathrm{d} B_1 (t)
$$
 \begin{equation}
\label{1.3}
~~~~~~~~~~~~~~~~~~~~~  \,+ \,{\,1- \zeta_1 \, \over 2} \,  \mathrm{d} L^{X_1 - X_2} (t) \,+ \,{\,1- \eta_1 \, \over 2} \,  \mathrm{d} L^{X_2 - X_1} (t) \,,
\end{equation}
 $$
\mathrm{d} X_2 (t)\,=\,  \Big( g \, \mathbf{ 1}_{ \{ X_1 (t) >  X_2 (t) \} } - h \, \mathbf{ 1}_{ \{ X_1 (t) \le  X_2 (t) \} } \Big) \, \mathrm{d} t \,+ \, \Big( \rho\, \mathbf{ 1}_{ \{ X_1 (t) \le  X_2 (t) \} }\, +\, \sigma\,  \mathbf{ 1}_{ \{ X_1 (t) >   X_2 (t) \} }\Big)\,      \mathrm{d} B_2 (t)
$$
\begin{equation}
\label{1.4}
~~~~~~~~~~~~~~~~~~~~~  \,+ \,{\,1- \zeta_2 \, \over 2} \,  \mathrm{d} L^{X_1 - X_2} (t) \,+ \,{\,1- \eta_2 \, \over 2} \,  \mathrm{d} L^{X_2 - X_1} (t)\,.
\end{equation}

\medskip
\noindent
Here and in the sequel  we denote  by  $\,  L^{X } (\cdot) \equiv L^{X } (\cdot \,;0)\,$  the   right-continuous local time accumulated at the origin by a generic continuous semimartingale $\, X(\cdot)\,$, by  $\,  L^{X }_- (\cdot) \equiv L^{-X } (\cdot \,;0)\,$  its   left-continuous version, and by $\, \widehat{L}^{X } (\cdot) = ( L^{X } (\cdot) + L^{X }_- (\cdot) ) / 2\,$ its symmetric  version; we collect in section \ref{sec2} the necessary reminders from the theory of semimartingale local time.

\smallskip
Each time the   particles collide, their trajectories are ``dragged'' by amounts proportional to the  right  local times,  $\, L^{X_1 - X_2} (t)\,$ and $\, L^{X_2 - X_1} (t)\,$, respectively,  that have been accumulated up to that instant $\,t\,$ at the origin by the differences $\, X_1 (\cdot) - X_2 (\cdot)\,$ and $\, X_2 (\cdot) - X_1 (\cdot)\,$; this is the significance of the last two terms in each of (\ref{1.3}), (\ref{1.4}). With the notation
\begin{equation}
\label{1.6}
   \zeta \,:=\, 1 + {\,  \zeta_1 - \zeta_2 \, \over 2} \,, \qquad   \eta \,:=\, 1 - {\,  \eta_1 - \eta_2 \, \over 2}    \,,
\end{equation} 
the proportionality constants of these interactions, $\, \zeta_i\,$ and $\, \eta_i\,$ for $\, X_i (\cdot)\,$ ($i=1, 2$), will be assumed to satisfy the conditions
\begin{equation}
\label{1.5}
  \zeta  + \eta \, \neq \, 0 \,,\qquad 0 \, \le \, \alpha \,:=\,{ \eta \over \, \eta + \zeta \,} \, \le \, 1\,.
\end{equation} 

We shall discuss in detail the significance of these conditions for the system of equations (\ref{1.3}), (\ref{1.4}); in particular, the fact that they are not only sufficient but also necessary for the well-posedness of the above system of stochastic equations. For the time being, let us note that in the special case $\, \zeta_1 =\eta_1 =1\,$,   the trajectory $\, X_1(\cdot)\,$ of the first particle  crosses the trajectory $\, X_2(\cdot)\,$ of the second particle without ``feeling'' it, that is, without being subjected to any local time drag; as we shall see in subsection \ref{Rem2}, we obtain this same effect under the more general condition  (\ref{1.8}). Likewise, the second particle does not ``feel'' the first, when $\, \zeta_2 =\eta_2 =1\,$ or, more generally, under the condition (\ref{1.9}). When $\, \zeta_i =\eta_i =1\,$  $(i=1, 2)$,  the local times vanish completely from (\ref{1.3}), (\ref{1.4}) and we are in the situation studied in detail by \textsc{Fernholz et al.} \cite{FIKP2012}.  In this case, the collisions of the   particles are totaly frictionless. 

At the other extreme $\, \zeta =0\neq \eta\,$ (respectively, $\, \eta =0\neq \zeta\,$)  the trajectory $\, X_1(\cdot)\,$ of the first particle bounces off the trajectory $\, X_2(\cdot)\,$ of the second particle (resp., the other way round), as if the latter trajectory  were a perfectly reflecting boundary; cf. subsection \ref{Rem1}. Think of the second (resp., the first) particle as being ``heavy'', so that in collisions with the ``light'' first (resp., second) particle its motion is unaffected, while the light particle undergoes perfect reflection. 

In between, for other values of the parameters, we have collisions that are neither totally frictionless (without local time drag), nor perfectly reflecting, but ``elastic'': The particles are subjected in general to local time drag, and this kind of friction manifests itself in an asymmetric fashion -- due to the presence of both right- and left- local times at the origin $\,  L^{Y} (\cdot) \equiv L^{Y} (\cdot \,;0+ )\,$ and $\,  L^{Y}_- (\cdot) \equiv L^{Y} (\cdot \,;0- )\,$ in (\ref{1.3}), (\ref{1.4}) for the difference $\, Y(\cdot) = X_1 (\cdot) - X_2 (\cdot)\,$. We call such collisions ``skew-elastic''.

\subsection{Preview}
\label{sec1.1}

\noindent
Under the  conditions of (\ref{1.5}),  the system of equations (\ref{1.3}), (\ref{1.4}) will be shown in section \ref{sec4} to admit a  weak solution, which is unique in the sense of the probability distribution; cf.$\,$Theorem \ref{Theorem 1}. Using a common terminology: under the conditions of (\ref{1.5}),  the system of equations (\ref{1.3}), (\ref{1.4}) is {\it well posed}. 

A crucial r\^ole in establishing this result will be played by the properties of the difference $\, Y (\cdot) = X_1 (\cdot) - X_2 (\cdot)\,$, for which we  show that 
$$
 W(t)\,:=\, Y (t) - ( x_1 - x_2)+ \lambda \int_0^t  \mathrm{sgn}  \big( Y(s)\big)\,   \mathrm{d} s  - 2\, \big(2\,\alpha -1\big) \, \widehat{L}^{\,Y} (t)    \,, \qquad 0 \le t < \infty 
$$
is standard Brownian motion $\, W(\cdot)\,$ in the notation of (\ref{1.1}), (\ref{1.5}). To put a little differently: we identify this difference $\, Y (\cdot) = X_1 (\cdot) - X_2 (\cdot)\,$ as a  so-called {\it Skew Brownian Motion with Bang-Bang drift}, a process that we study in detail in Section \ref{sec5}. 

\smallskip
Similarly, recalling the notation of (\ref{1.6}) and setting 
  $$
  \beta \,:=\, {\,  \eta\, (\zeta_1 + \zeta_2) + \zeta  (\eta_1 + \eta_2)   \, \over \, 2\,( \eta + \zeta )\,}\,,
  $$
we   identify the process 
$$  
V(t)\,:=\, X_1  (t) + X_2 ( t) - (x_1 + x_2) - (g-h)\,t - 2\, (1-\beta)\,\widehat{L}^{\,Y} (t) \, , \qquad 0 \le t < \infty\,
$$  
as another standard Brownian motion, whose cross-variation with the Brownian motion $\,W(\cdot)\,$ is 
$$
\langle V,W \rangle (\cdot) \,=\, \langle V,Y \rangle (\cdot) \,=\, \gamma   \int_0^{\, \cdot} \mathrm{sgn} \big( Y(t) \big) \, \mathrm{d} t\,, \qquad \gamma  \,:=\, \rho^2 - \sigma^2 \,.
$$
These identifications allow us then to represent the motions $\, X_1 (\cdot)\,$,  $\, X_2 (\cdot)\,$ of the individual particles in the form 
  \[
X_1 (t) = x_1 + \mu\, t + \rho^2  \big(  Y^+ (t) - y^+ \big) - \sigma^2 \big( Y^- (t) - y^- \big) + \big(   1    -   \beta - \gamma \big)\, \widehat{L}^{\,Y} (t) + \rho \, \sigma \, Q (t)  \,,
\]
\[
X_2 (t) = x_2 + \mu\, t - \sigma^2 \,\big(  Y^+ (t) - y^+ \big) + \rho^2  \big(  Y^- (t) - y^- \big) + \big(   1   -   \beta - \gamma \big)\, \widehat{L}^{\,Y} (t)   +    \rho \, \sigma \, Q (t)  \,;
\]

\noindent
 here  $\,Q(\cdot)\,$ is yet another standard Brownian motion, independent of the difference $\, Y (\cdot) = X_1 (\cdot) - X_2 (\cdot)\,$, and $$\, \mu \,= \,g \, \rho^2 - h\, \sigma^2\,.$$ This way we construct a weak solution to the system of equations (\ref{1.3}), (\ref{1.4}), and also show that uniqueness in distribution holds for it. 
 
Always under the  conditions of (\ref{1.5}),  the system of equations (\ref{1.3}), (\ref{1.4}) is shown in section \ref{sec4} actually to admit a {\it pathwise unique, strong solution;} cf.$\,$Theorem \ref{Theorem 2}. Here we     refine  the \textsc{Le$\,$Gall} \cite{MR0770393} \cite{MR777514}  
methodology,  that we used in the recent work \textsc{Fernholz et al.} \cite{FIKP2012} 
 to establish pathwise uniqueness for a generalization of the perturbed \textsc{Tanaka} equation of \textsc{Prokaj} \cite{tanaka2009}.

\smallskip
In fact,  {\it  the  conditions in (\ref{1.5}) turn out to be not just sufficient but also necessary for the well-posedness of the system   (\ref{1.3}), (\ref{1.4});} cf. Proposition \ref{Proposition 3}. As we shall see in Remarks \ref{EZ0} and \ref{EZ1}, this system admits no solution in the case $\, \eta = - \zeta \neq 0\,$; whereas it has lots of solutions, i.e., uniqueness in distribution fails  for the system of equations (\ref{1.3}) and (\ref{1.4}), when $\, \eta = \zeta =0\,$.  Finally, if we do have $\, \eta + \zeta \neq 0\,$ yet (\ref{1.5}) fails because 
  $\, \alpha \notin [0,1]\,$, it is seen  in Remark \ref{HS} that  the system   (\ref{1.3}), (\ref{1.4}) once again fails to admit a solution.

\smallskip
Section \ref{sec6} discusses some special configurations of the parameters $\, \eta_i, \, \zeta_i\,$ ($i=1,2$) in (\ref{1.3}), (\ref{1.4}) that give rise to some rather interesting structure.  We see, in particular in the non-degenerate case $\, \rho \, \sigma >0\,$, that when $\, \beta =0\,$ (respectively, $\, \beta =2\,$), the trajectory $\, X_1 (\cdot) \vee X_2 (\cdot)\,$ of the ``leader'' (respectively, $\, X_1 (\cdot) \wedge X_2 (\cdot)\,$ of the ``laggard'') is   Brownian motion with drift, with perfect reflection on the trajectory $\, X_1 (\cdot) \wedge X_2 (\cdot)\,$ of the ``laggard'' (respectively, $\, X_1 (\cdot) \vee X_2 (\cdot)\,$ of the ``leader''), which is then another,  independent Brownian motion with drift.

\smallskip
Section \ref{sec5} develops the theory and properties of the Skew Brownian Motion with Bang-Bang drift. Finally, Section \ref{sec7}    uses these properties to compute  the transition probabilities and the time-reversal of the planar diffusion $\, (X_1(\cdot), \, X_2 (\cdot))\,$. 


\subsection{Extant Work and Open Questions}
\label{sec1.2}

\noindent
The study of multidimensional stochastic differential equations that involve  a local time supported
on a smooth hypersurface starts with the work of \textsc{Anulova}  
\cite{MR609190} and \textsc{Portenko} \cite{MR0440716} 
\cite{MR522237} 
\cite{MR532446} 
\cite{MR1104660}. 
To the best of our knowledge, systems of stochastic  equations of the type 
\begin{equation}
\label{SZV}
X_i (\cdot) \,=\, x_i + B_i (\cdot) + \sum_{j\neq i }\, \,q_{ij}\, \widehat{L}^{X_i - X_j} (\cdot)\,, \qquad i=1, \cdots, n
\end{equation}
for a suitable array of real constants $\,  ( q_{i j}  )_{1 \le i, j \le n}\,$, with $\, B_1 (\cdot), \cdots, B_n (\cdot)\,$ independent standard Brownian motions, were studied first by \textsc{Sznitman \& Varadhan} \cite{MR833269}. 
In fact, these authors consider the more general model 
\begin{equation}
\label{SZV1}
{\bm X}(t) \, =\, {\bm x} +   
{\bm B}(t) + \sum_{k=1}^{N} \,{\bm q}_{k}\,    
\widehat{L}^{\,{\bm n}_{k} \cdot {\bm X}}(t) \,, \quad \, 0 \le t < \infty \,,
\end{equation}
 where $\,{\bm X}(\cdot) \, :=\, (X_{1}(\cdot), \ldots , X_{n}(\cdot) )'\,$, $\,{\bm B}(\cdot) \, :=\, (B_{1}(\cdot), \cdots , B_{n}(\cdot))'\,$ is     Brownian motion in $\, \R^n\,$, $\, {\bm x} \in \R^n\,$,  the unit column vectors $\,{\bm n}_{k}\,$  generate pairwise distinct hyperplanes, and the column vectors  $\,{\bm q}_{k}\,$  satisfy  the orthogonality conditions $\,\,{\bm q}_{k}\, \cdot \,{\bm n}_{k}\, \, =\, 0\, $ for $\,k = 1, \ldots , N\,$. When $\,g \, =\, h \, =\, 0\,$ and  $\,\sigma \, =\,  \rho \, \,$,    it can be verified -- using the relationships  (\ref{alpha}) between the symmetric local time and the right local time -- that the system (\ref{1.3})-(\ref{1.4}) is equivalent to the  model (\ref{SZV1}) with parameters 
$\, n \, =\, 2\, $, $\, N \, =\,  1\,$, 
$$\,
 {\bm n}_{1} \, :=\, (1, -1)' \, / \, \sqrt{2} \,, \quad \hbox{and} \quad  
{\bm q}_{1} \, :=\, \big(   \alpha ( 1- \zeta_{1}) + (1-\eta_{1}) (1-\alpha)  \, , \,
  \alpha ( 1- \zeta_{2}) + (1-\eta_{2}) (1-\alpha)  \big)' \, .
  $$  
The orthogonality conditions amount then to $\, \eta \, =\, \zeta\,$. Thus, we can apply the results of \textsc{Sznitman \& Varadhan} \cite{MR833269}, 
 if $\,g \, =\, h \, =\, 0\,$, $\,\sigma \, =\,  \rho \,$, $\,\eta \, =\, \zeta \,$ in our system (\ref{1.3})-(\ref{1.4}).  

There are rather obvious similarities, as well as differences, between the system (\ref{SZV}) and that of (\ref{1.3}), (\ref{1.4}). In particular, it would be very interesting to extend the results of this paper to systems of stochastic differential equations of the type
$$
 X_i (\cdot)\,=\,x_i +  \sum_{k=1}^n \int_0^{\, \cdot}\,  \delta_k \, \mathbf{1}_{\{ X_i (t) = X_{(k)} (t) \}}\, \mathrm{d}t \,+\, \sum_{k=1}^n \int_0^{\, \cdot}\,\,  \sigma_k \, \mathbf{1}_{\{ X_i (t) = X_{(k)} (t) \}}\, \mathrm{d} B_i (t)
 $$
\begin{equation}
\label{BaPar}
~~~~~~~~~~~~~~~~~~~~+\,\sum_{j\neq i} \Big( \,q^+_{ij}\, L^{X_i - X_j} (\cdot) + q^-_{ij}\, L^{X_j - X_i} (\cdot) \Big) \,, \qquad i=1, \cdots, n 
\end{equation}
for an arbitrary number $\, n \in \N\,$ of particles, with  $\, x_1, \cdots, x_n\,$ and $\, \delta_1, \cdots, \delta_n\,$   given real constants, with   $\, \sigma_1, \cdots, \sigma_n\,$   given positive constants, with suitable arrays  $\,  ( q^{\,\pm}_{i j}  )_{1 \le i, j \le n}\,$ of real constants,    the ``descending order statistics'' notation
 $$
 \max_{1 \le j \le n} X_j (t) \,=:\, X_{(1)}(t)  \ge  X_{(2)}(t)  \ge \, \cdots\,  \ge X_{(n-1)}(t)  \ge X_{(n)}(t) \,:=\, \min_{1 \le j \le n} X_j (t)\,,
$$
and lexicographic breaking of ties. This system (\ref{BaPar}) exhibits both features of rank-dependent characteristics and skew-elastic collisions that are manifest in (\ref{1.3}), (\ref{1.4}), but involves several particles rather than just two.



\section{On Semimartingale Local Time}
\label{sec2}

\noindent
Let us recall the notion of a continuous, real-valued semimartingale   
 \begin{equation}
\label{Semi}
   X   (\cdot)\, =\,  X (0) + M(\cdot) +C(\cdot) \,,
  \end{equation}
      where $\, M (\cdot)\,$ is a continuous local martingale and $\, C (\cdot)\,$ a continuous process of finite first variation       such that $\, M(0) = C(0) =0\,$. The {\it local time} $\,  L^{ X } (t; \xi)\,$ accumulated at a given ``site'' $\, \xi \in \R\,$  over the time-interval $\, [0,t]\,$ by this process,  is    
      $$
      L^{ X}  (t; \xi)\,:=\,\lim_{\varepsilon \downarrow 0} \, { 1 \over \,2 \, \varepsilon\,} \int_0^t \, \mathbf{1}_{ \{ \xi \le     X    (s)  <  \xi + \varepsilon \} } \, \mathrm{d} \langle X \rangle ( s)~~~~~~~~~~~~~~~~
      $$
     \begin{equation}
\label{Tanaka}
~~~~~~~~~~~~~~~~~~~~~~~~~~~~~~~~~~~~~ = \,  \big(  X  (t) - \xi \big)^+ - \big(  X  (0) - \xi \big)^+  - \int_0^t \mathbf{ 1}_{\{  X (s) > \xi \} }\, \mathrm{d}   X   (s)  \,,
\end{equation}

\noindent
where $\, \langle X \rangle ( \cdot) \equiv \langle M \rangle ( \cdot)\,$. For every fixed $\, \xi \in \R\,$ this defines a nondecreasing, continuous and adapted process $\, L^{ X}  ( \cdot\,; \xi )\,$  which is flat off the   set     $\, \{ t \ge 0 :  X ( t) = \xi \} \,$, namely
\begin{equation}
\label{flat}
\int_0^\infty \, \mathbf{ 1}_{ \{  X  (t) \neq \xi\} }\, \mathrm{d}   L^{  X}  ( t; \xi) \,=\,0 \,   \,; \quad \hbox{and we  have also the property} ~~   \int_0^\infty \, \mathbf{ 1}_{ \{   X  (t) = \xi\} }\, \mathrm{d} \langle X \rangle (t) \,=\, 0\,.
\end{equation}
On the other hand, for each fixed $\, T \in (0, \infty)\,$ the mapping $\, \xi \mapsto L^X (T\,; \, \xi \,)\,$ is almost surely RCLL (Right-Continuous on $[0, \infty)$, with Limits from the Left on $(0, \infty)$) paths, and has jumps of size
 \begin{equation}
\label{jump}
  L^{  X}  (T; \xi) - L^{  X }  (T; \xi-)\,=  \,\int_0^T \, \mathbf{ 1}_{ \{  X  (t) = \xi \} }\, \mathrm{d}    X (t) \,=\, \int_0^T \, \mathbf{ 1}_{ \{   X  (t) = \xi\} }\, \mathrm{d}   C (t) \,.
  \end{equation}
We shall employ also the notation
   \begin{equation}
\label{symm}
\widehat{L}^{  X}  (T; \xi)  \,:=\, { 1 \over \,2\,} \left( L^{  X}  (T; \xi) + L^{  X }  (T; \xi-) \right)   \end{equation}
  for the so-called ``symmetric local time" accumulated at the site $\, \xi\,$ over the time interval $  [0,T]\,$.    For these local times  we prefer to use the simpler notation 
    \begin{equation}
\label{NOT} 
 L^{  X}  (\cdot )\, \equiv \, L^{  X}  (\cdot\, ; 0)\,, \qquad  L^{  X}_-  (\cdot )\, \equiv \, L^{  X}  (\cdot\, ; 0-)\,, \qquad   \widehat{L}^{  X}  (\cdot ) \equiv  \widehat{L}^{  X}  (\cdot\, ; 0) 
  \end{equation}
when we evaluate them at the origin $\, \xi =0\,$,  and note 
      \begin{equation}
\label{NOT1}
L^{  X}_-  (\cdot )\, = \, L^{ -X}  (\cdot   )\,=\,\lim_{\varepsilon \downarrow 0} \, { 1 \over \,2 \, \varepsilon\,} \int_0^{\, \cdot} \, \mathbf{1}_{ \{ 0 \ge   X    (t)  >  -  \varepsilon \} } \, \mathrm{d} \langle X \rangle ( t)\,.
  \end{equation}
  Finally, we recall the occupation time density formulae
 \begin{equation}
\label{OTD}  
\int_0^{\,\cdot} h (X(t))\, \de \langle X \rangle (t)\,=\,   2 \int_\R L^X (\cdot\,;  \xi)\, h (\xi)\, \de \xi  \,=\,   2 \int_\R \widehat{L}^X (\cdot\,;  \xi)\, h (\xi)\, \de \xi \,,
\end{equation}
valid for every Borel measurable function $\, h : \R \ra [0, \infty)\,$, as well as the \textsc{It\^o-Tanaka} formulae
 \begin{equation}
\label{ITOT1} 
f(X(\cdot)) \,=\, f(X(0)) + \int_0^{\,\cdot} D^- f (X(t))\, \de X(t) + \int_\R L^X (\cdot\,;  \xi) \, f^{\prime \prime} (\de \xi)\,,
\end{equation}
 \begin{equation}
\label{ITOT2} 
f(X(\cdot)) \,=\, f(X(0)) + {1 \over \,2\,} \int_0^{\,\cdot} \big( D^+ f (X(t)) + D^- f (X(t))\big) \, \de X(t) + \int_\R \widehat{L}^X (\cdot\,; \xi) \, f^{\prime \prime} (\de \xi)\,.~~~
\end{equation}
Here $\, f: \R \ra \R\,$ is the difference of two convex functions, $\, D^{\pm} f (\cdot)\,$ denote its derivatives from left and right, and $\, f^{\prime \prime} (\cdot)\,$ denotes its second derivative measure. 

\medskip

\subsection{Tanaka Formulae}
 \label{sec2a}

\noindent
For a continuous, real-valued  semimartingale $\, X(\cdot)\,$ as in (\ref{Semi}), and with the    conventions
\begin{equation}
\label{sgn}
  \overline{\text{sgn}} \,(x) \,:=\, \mathbf{ 1}_{(0,  \infty)} (x)  - \mathbf{ 1}_{(- \infty, 0)} (x)   \,, \quad  
   \text{sgn}  \,(x) \,:=\, \mathbf{ 1}_{(0,  \infty)} (x)  - \mathbf{ 1}_{(- \infty, 0]} (x)\,, \qquad x \in \mathbb{R}
\end{equation}
for the symmetric and the left-continuous versions of the signum function, we obtain from (\ref{ITOT1}),  (\ref{ITOT2})  the \textsc{Tanaka} formulae
 \begin{equation}
\label{Tanaka2}
  | X (\cdot) - \xi| \,=\,   | X (0) - \xi| + \int_0^\cdot \text{sgn} \big( X(t) - \xi\big) \, \mathrm{d} X(t) \,+\, 2\, L^X (\cdot\, ; \xi) 
\end{equation}
\begin{equation}
\label{Tanaka1}
 ~~~~~~~~~~~~~~~~~~~~ \,=\,   | X (0) - \xi| + \int_0^\cdot \overline{\text{sgn}} \big( X(t) - \xi \big) \, \mathrm{d} X(t) \,+\,  2\,  \widehat{L}^{X} (\cdot\,; \xi)\,. 
\end{equation}
 Applying (\ref{Tanaka2}) with $\, \xi =0\,$ to the continuous, nonnegative semimartingale $\, | X(\cdot)|\,$, then comparing with the expression of (\ref{Tanaka2}) itself, we obtain the companion 
\begin{equation}
\label{2.19.c}
2\, L^{  X}  (\cdot ) - L^{  |X|}  (\cdot)\,=\, \int_0^\cdot \, \mathbf{ 1}_{ \{  X  (t) = 0 \} }\, \mathrm{d}    X (t) \,=\, \int_0^\cdot \, \mathbf{ 1}_{ \{   X  (t) = 0 \} }\, \mathrm{d}   C (t) 
\end{equation}
of the property (\ref{jump}). For  the theory that undergirds these  results we refer, for instance, to   \textsc{Karatzas \& Shreve} \cite{MR1121940}, 
section 3.7.


\section{Analysis}
\label{sec3}

\noindent
Let us suppose that such a probability space as stipulated in section \ref{sec1} has been constructed, and on it a pair $\, B_1 (\cdot)\,$, $  B_2 (\cdot)\,$ of   independent standard Brownian motions, as well as two continuous, nonnegative semimartingales $\,  X_1(\cdot), \,X_2 (\cdot) \,$   such that the dynamics (\ref{1.3})-(\ref{1.4})   are satisfied. We import the notation of \textsc{Fernholz  et al.} \cite{FIKP2012}  
: in addition to (\ref{1.1}), we set 
\begin{equation}
\label{2.1}
 \nu \,=\,   g-h\,, \quad y \,=\, x_1 - x_2 \,, \quad z \,:=\, x_1 + x_2 >0\,, \quad  r_1\,=\, x_1 \vee x_2\,, \quad  r_2\,=\, x_1 \wedge x_2\,,
\end{equation}
and introduce the difference and the sum of the two component processes, namely
\begin{equation}
\label{2.2}
Y (\cdot ) \,:=\,   X_1 (\cdot) - X_2 (\cdot)\,, \qquad Z (\cdot )  \,:=\,   X_1 (\cdot) + X_2 (\cdot)\,.
\end{equation}

\subsection{Auxiliary Brownian Motions}
\label{sec3.4}

\noindent
We introduce also the two planar Brownian motions $\, \big( W_1 (\cdot), W_2 (\cdot) \big) $ and  $\, \big( V_1 (\cdot), V_2 (\cdot) \big) $,  given  respectively by
\begin{align}
\label{2.3}
W_1 (\cdot) \,&:=\, \int_0^\cdot \mathbf{ 1}_{ \{ Y (t) > 0 \} } \, \mathrm{d} B_1 (t)
- \int_0^\cdot \mathbf{ 1}_{ \{ Y (t) \le 0 \} } \, \mathrm{d} B_2 (t) \,, \\ 
\label{2.4}
W_2 (\cdot) \,&:=\, \int_0^\cdot \mathbf{ 1}_{ \{ Y (t) \le 0 \} } \, \mathrm{d} B_1 (t) - \int_0^\cdot \mathbf{ 1}_{ \{ Y (t) > 0 \} } \, \mathrm{d} B_2 (t)   
\end{align}
and 
\begin{align}
\label{2.5}
V_1(\cdot)\,  &:=\,  \int_0^\cdot    \mathbf{ 1}_{ \{ Y (t) > 0 \} } \, \mathrm{d} B_1
(t) +   \int_0^\cdot \mathbf{ 1}_{ \{ Y (t) \le 0 \} }  \, \mathrm{d} B_2 (t)
\,, \\
\label{2.6}
V_2(\cdot)  \,&:= \, \int_0^\cdot    \mathbf{ 1}_{ \{ Y (t) \le 0 \} } \, \mathrm{d} B_1 (t) + \int_0^\cdot  \mathbf{ 1}_{ \{ Y (t) > 0 \} }  \, \mathrm{d} B_2 (t)   \,.
\end{align}

\noindent
 Finally, we construct the      Brownian motions $\,  W  (\cdot)  $, $\,  V  (\cdot)  \,$,  $\,     Q (\cdot)\,$ and $\, W^\flat (\cdot) \,   $,  $\, V^\flat (\cdot) \,   $, $\, U^\flat (\cdot) \,   $  as
\begin{equation}
\label{2.7}
W (\cdot) \,:=\, \rho\, W_1 (\cdot) \,+\, \sigma \, W_2 (\cdot)\,,  \quad 
V (\cdot) \,:=\, \rho\, V_1 (\cdot) \,+\, \sigma \, V_2 (\cdot)\,, \quad Q (\cdot) :=\sigma \, V_1 (\cdot)+ \rho\, V_2 (\cdot)  \,,~
\end{equation}
\begin{equation}
\label{2.9.b}
W^{\,\mathbf{ \flat}}( \cdot) := \rho\, W_1 (\cdot) - \sigma \, W_2 (\cdot)\,, \quad  V^{\,\mathbf{ \flat}}( \cdot) := \rho\, V_1 (\cdot) - \sigma \, V_2 (\cdot)\,,\quad U^\flat (\cdot) :=\sigma \, W_1 (\cdot)-  \rho\, W_2 (\cdot) \,;
\end{equation}
 we note the independence of  $\,     Q (\cdot)\,$ and $\, W (\cdot)  \,  $, the independence of  $\,     Q (\cdot)\,$ and $\, V^\flat (\cdot)  \,  $, and observe  the intertwinements
\begin{equation}
\label{2.10}
V_j(\cdot) = (-1)^{j+1}\int_0^{\,\cdot}  \text{sgn}  \big( Y (t) \big)   \mathrm{d} W_j (t)   \quad    (j=1, 2) 
\,, \qquad   V^\flat(\cdot)   =  \int_0^{\,\cdot}   \text{sgn}  \big( Y (t) \big)   \mathrm{d}  W (t)   
\end{equation}
and
\begin{equation}
\label{2.10.a}
    V(\cdot)   =  \int_0^{\,\cdot}   \text{sgn}  \big( Y (t) \big)   \mathrm{d}  W^\flat (t)\,, \qquad  Q(\cdot)   =  \int_0^{\,\cdot   }\text{sgn}  \big( Y (t) \big)   \mathrm{d}  U^\flat (t)\, .
\end{equation}

\subsection{The Difference and the Sum}
\label{sec3.3}

\noindent
 After this preparation,  we  observe that the difference $\, Y(\cdot)\,$ and the sum $\, Z  (\cdot)\,$ from (\ref{2.2}) satisfy, respectively, the stochastic integral equation 
\begin{equation}
\label{2.11}
Y (\cdot) \,=\, y - \lambda \int_0^{\, \cdot}  \mathrm{sgn}  \big( Y(t)\big)\,   \mathrm{d} t  + (1-\zeta) \, L^Y (\cdot) - (1-\eta) \, L^Y_- (\cdot ) + W(\cdot)    
\end{equation}
which involves both the right- and the left- local time at the origin of its  solution process $\, Y(\cdot)\,$, and the identity 
\begin{equation}
\label{2.12}
 Z (t) \,=\,z + \nu \,t +  V (t) +   \big(1-\overline{\zeta} \,\big) \, L^Y (t ) + \big(1-\overline{\eta}\, \big) \, L^Y_- (t )\,, \qquad 0 \le t < \infty\,.  
\end{equation}

\medskip
\noindent
We have used here the notation of (\ref{1.1}), (\ref{2.1}), (\ref{1.6}), as well as 
\begin{equation}
\label{1.7}
     \overline{\zeta} \,:=\,   {\,  \zeta_1 + \zeta_2 \, \over 2} \,, \qquad \overline{\eta} \,:=\,   {\,  \eta_1 + \eta_2 \, \over 2}  \,.
\end{equation}   
We note from (\ref{flat}) and (\ref{2.11}) that $\, Y(\cdot)\,$ is a continuous semimartingale with 
\begin{equation}
\label{ZS}
\int_0^\infty \mathbf{ 1}_{ \{ Y(t)=0 \}  }\, \de t\,=\, \int_0^\infty \mathbf{ 1}_{ \{ Y(t)=0 \}  }\, \de \langle Y \rangle (t)\,=\, 0 \,, \qquad \mathrm{a.s.},
\end{equation}  
and that on the strength of  (\ref{jump}), (\ref{flat})  we have
$$
L^{Y}  (\cdot\,; 0) - L^{ Y}  (\cdot\,; 0-)\,=  \,\int_0^{\,\cdot} \, \mathbf{ 1}_{ \{  Y  (t) = 0 \} }\, \big( (1-\zeta)\, \mathrm{d}   L^Y (t; 0)  - (1-\eta)\, \mathrm{d}   L^Y (t; 0-) \big) 
$$
$$
~~~~~~~~=\,  (1-\zeta)\,     L^Y (\cdot\,; 0)  - (1-\eta)\,    L^Y (\cdot\,; 0-)
$$
or equivalently 
\begin{equation}
\label{LR}
\zeta \, L^{Y}  (\cdot ) \,=\, \eta \, L^{Y}_-  (\cdot\, ) \,.
\end{equation}

\medskip
\noindent
From this relationship and (\ref{flat})-(\ref{symm}), (\ref{2.19.c})   we obtain
\begin{equation}
\label{alpha}
2\, \widehat{L}^Y (\cdot)\,=\, L^{|Y|} (\cdot)\qquad \mathrm{and} \qquad 
 L^Y (\cdot)\,=\, \alpha \,L^{|Y|} (\cdot)\,, ~~~ L^Y_- (\cdot)\, = \, (1 - \alpha) \,  L^{|Y|} (\cdot)\,,
\end{equation}
where we introduce as in (\ref{1.5}) the ``skewness parameter" 
\begin{equation}
\label{SP}
\alpha \,:=\, { \eta \over \, \eta + \zeta\,} \,.
\end{equation}
$\bullet~$ 
With this notation,  and recalling (\ref{ZS}) and (\ref{alpha}), we see that equation (\ref{2.11}) takes the form 
\begin{equation}
\label{SBBBM}
Y (\cdot) \,=\, y - \lambda \int_0^{\,\cdot}  \overline{\text{sgn}}  \big( Y(t)\big)\,   \mathrm{d} t \, +\, 2\, \big(2\,\alpha -1\big) \, \widehat{L}^{\,Y} (\cdot) \, + W(\cdot)  
\end{equation}
of the equation  for a {\it Skew Brownian Motion with Bang-Bang drift} (Skew Bang-Bang Brownian Motion, or SBBBM for short). This is a very close relative of the Skew Brownian motion, that was  introduced by \textsc{It\^o \& McKean} \cite{MR0154338}
, \cite{MR0345224} 
and was furhter studied by \textsc{Walsh} \cite{TempsLocaux78}
, \textsc{Harrison \& Shepp} \cite{MR606993} 
; see \textsc{Lejay} \cite{MR2280299} 
for a comprehensive survey.  

\smallskip
The diffusion process $\, Y(\cdot)\,$ of (\ref{SBBBM}) is studied in detail in section \ref{sec5}. It is a strong \textsc{Markov} and \textsc{Feller} process, whose transition probabilities can be computed explicitly; see (\ref{TDF1})-(\ref{TDF3}) below. In particular, it is shown in section \ref{sec5} that,  for $\, 0 \le \alpha \le 1\,$,   the stochastic equation (\ref{SBBBM}) has a pathwise unique, strong solution  and that    the filtration identities
\begin{equation}
\label{YW}
\mathfrak{ F}^Y (t) \,=\,\mathfrak{ F}^W (t)\,, \qquad \forall ~~ t \in [0, \infty)
\end{equation}
hold. Here and in what follows, given a process $\,\Xi  : [0, \infty) \times \Omega \rightarrow \R^d\,$ with values in some Euclidean space and RCLL paths, we shall  use the convention
$$
\mathbf{F}^{\,\Xi} = \{ \mathfrak{F}^{\,\Xi } (t) \}_{0 \le t < \infty }  
$$
for the smallest filtration  to which $\, \Xi (\cdot)\,$ is adapted that satisfies the ``usual conditions'' of right continuity and augmentation by sets of measure zero.  

\medskip
\noindent
$\bullet~$ 
Similarly, and with the notation of (\ref{SP}), the expression 
  (\ref{2.12}) takes the form 
\begin{equation}
\label{2.12.a}
X_1 (t) + X_2 (t)\,=\,  Z (t) \,=\,z + \nu \,t +  V (t) +  2\, \big(1-   \beta \big) \,\widehat{ L}^{\,Y} (t )  \,, \qquad 0 \le t < \infty   
\end{equation}
where, as in subsection \ref{sec1.1},  we set 
\begin{equation}
\label{SP1}
  \beta \,:=\, {\,  \eta\, \overline{\zeta} + \zeta\, \overline{\eta} \, \over \, \eta + \zeta\,}\,.
\end{equation}

\medskip
\begin{rem}
 \label{EZ0}
  It is clear from (\ref{LR}) that the equation (\ref{2.11}) can be written as 
\begin{equation}
\label{YLPM}
Y (\cdot) \,=\, y - \lambda \int_0^{\, \cdot} \text{sgn}  \big( Y(t)\big)\,   \mathrm{d} t  +  L^Y (\cdot)  - L^Y_- (\cdot)+ W(\cdot)\, .
\end{equation}
To wit:  skew Brownian motion with bang-bang drift solves the equation (\ref{YLPM}), for {\it any} value $\, \alpha \in [0,1]\,$ of its skewness parameter. We conclude that {\it uniqueness in distribution fails for this equation (\ref{YLPM}), thus also for the equation (\ref{2.11}) that governs the difference $\, Y (\cdot)= X_1 (\cdot) - X_2(\cdot)\,$ when $\, \eta = \zeta =0\,$.} 

In particular,   uniqueness in distribution cannot possibly hold for  the system (\ref{1.3}), (\ref{1.4})   in this case $\, \eta = \zeta =0\,$.
\end{rem}

\begin{rem}
 \label{EZ1}
 When $\, \eta = - \zeta \neq 0\,$ we get $\, L^Y  (\cdot) +  L^Y_- (\cdot) \equiv 0\,$ from  (\ref{LR}), thus 
 \begin{equation}
\label{ELT} 
L^Y  (\cdot) \, \equiv \, L^Y_- (\cdot)\, \equiv \,0\,,
\end{equation}
 and the equation (\ref{2.11}) takes the form of Brownian motion with bang-bang drift
 \begin{equation}
\label{KS}
Y (\cdot) \,=\, y - \lambda \int_0^{\, \cdot} \text{sgn}  \big( Y(t)\big)\,     \mathrm{d} t      + \,W(\cdot)    \,.
\end{equation}
This diffusion process was studied in detail  by \textsc{Karatzas \& Shreve} \cite{MR744236}
, who computed its transition probabilities and the joint distribution of the triple $\, \big(\,Y(t), \,L^Y (t), \,\int_0^t \mathbf{ 1}_{\{ Y(s)>0\}}\, \mathrm{d}s \big)\,$.   This diffusion {\it does} accumulate local time at the origin: indeed, on the strength of (\ref{jump}), (\ref{flat}), we have   
$$
L^Y (\cdot) - L^Y_- (\cdot) \,=\, \int_0^{\, \cdot} \mathbf{ 1}_{ \{ Y(t) =0\} } \, \mathrm{d} Y(t)\,=\, \lambda \int_0^{\, \cdot} \mathbf{ 1}_{ \{ Y(t) =0\} } \, \mathrm{d} t \,=\, \lambda \int_0^{\, \cdot} \mathbf{ 1}_{ \{ Y(t) =0\} } \, \mathrm{d} \langle W \rangle ( t)\,=\,0 
$$
almost surely, but also  $\, \p ( L^Y  ( t) = L^Y_- ( t) >0)>0\,$   for every $\, t \in (0, \infty)\,$; this    contradicts (\ref{ELT}). In fact, we have $\, \p ( L^Y  ( t) = L^Y_- ( t) >0)=1\,$  for $\, y =0\,$. 

\smallskip
We conclude that {\it the equation (\ref{2.11}) for the difference $\, Y (\cdot)= X_1 (\cdot) - X_2(\cdot)\,$ has no solution in the case} $\, \eta = - \zeta \neq 0\,$. Thus,  the system (\ref{1.3}), (\ref{1.4}) cannot possibly have a solution in this case.
 \end{rem}

\subsection{Auxiliary Systems}
\label{sec3.5}

\noindent
From the equations of (\ref{2.11}), (\ref{2.12}) and using the notation in (\ref{2.2})-(\ref{2.7}), we obtain a system of stochastic differential equations
$$
\mathrm{d} X_1 (t)\,=\,  \Big( g \, \mathbf{ 1}_{ \{ X_1 (t) \le  X_2 (t) \} } - h \, \mathbf{ 1}_{ \{ X_1 (t) >  X_2 (t) \} } \Big) \, \mathrm{d} t \,+ \,   \rho\, \mathbf{ 1}_{ \{ X_1 (t) >  X_2 (t) \} }\, \mathrm{d} W_1 (t) ~~~~~~~~~~~~~~~~
$$
 \begin{equation}
\label{1.3.a}
~~~~~~~~~~~~~~+\, \sigma\,  \mathbf{ 1}_{ \{ X_1 (t) \le   X_2 (t) \} }\,      \mathrm{d} W_2 (t)   \,+ \,{\,1- \zeta_1 \, \over 2} \,  \mathrm{d} L^{X_1 - X_2} (t) \,+ \,{\,1- \eta_1 \, \over 2} \,  \mathrm{d} L^{X_2 - X_1} (t) \,,
\end{equation}
 $$
\mathrm{d} X_2 (t)\,=\,  \Big( g \, \mathbf{ 1}_{ \{ X_1 (t) >  X_2 (t) \} } - h \, \mathbf{ 1}_{ \{ X_1 (t) \le  X_2 (t) \} } \Big) \, \mathrm{d} t \,- \, \rho\, \mathbf{ 1}_{ \{ X_1 (t) \le  X_2 (t) \} }\, \mathrm{d} W_1 (t)
~~~~~~~~~~~~~$$
\begin{equation}
\label{1.4.a}
~~~~~~~~~~~ - \, \sigma\,  \mathbf{ 1}_{ \{ X_1 (t) >   X_2 (t) \} }\,      \mathrm{d} W_2 (t)  \,+ \,{\,1- \zeta_2 \, \over 2} \,  \mathrm{d} L^{X_1 - X_2} (t) \,+ \,{\,1- \eta_2 \, \over 2} \,  \mathrm{d} L^{X_2 - X_1} (t)\,,
\end{equation}

\medskip
\noindent
quite similar to that of (\ref{1.3}), (\ref{1.4}), but now driven by the planar Brownian motion $\, (W_1 (\cdot), W_2 (\cdot))\,$. In a totally analogous manner, we obtain also the system  
$$
\mathrm{d} X_1 (t)\,=\,  \Big( g \, \mathbf{ 1}_{ \{ X_1 (t) \le  X_2 (t) \} } - h \, \mathbf{ 1}_{ \{ X_1 (t) >  X_2 (t) \} } \Big) \, \mathrm{d} t \,+ \,   \rho\, \mathbf{ 1}_{ \{ X_1 (t) >  X_2 (t) \} }\, \mathrm{d} V_1 (t) 
~~~~~~~~~~~~~~~~~~~~~$$
 \begin{equation}
\label{1.3.b}
~~~~~~~~~+\, \sigma\,  \mathbf{ 1}_{ \{ X_1 (t) \le   X_2 (t) \} }\,      \mathrm{d} V_2 (t)  \,+ \,{\,1- \zeta_1 \, \over 2} \,  \mathrm{d} L^{X_1 - X_2} (t) \,+ \,{\,1- \eta_1 \, \over 2} \,  \mathrm{d} L^{X_2 - X_1} (t) \,,
\end{equation}
 $$
\mathrm{d} X_2 (t)\,=\,  \Big( g \, \mathbf{ 1}_{ \{ X_1 (t) >  X_2 (t) \} } - h \, \mathbf{ 1}_{ \{ X_1 (t) \le  X_2 (t) \} } \Big) \, \mathrm{d} t \,+ \, \rho\, \mathbf{ 1}_{ \{ X_1 (t) \le  X_2 (t) \} }\, \mathrm{d} V_1 (t)
~~~~~~~~~~~~~~~~~~$$
\begin{equation}
\label{1.4.b}
~~~~~~~~~~~~~~~+\, \sigma\,  \mathbf{ 1}_{ \{ X_1 (t) >   X_2 (t) \} }\,      \mathrm{d} V_2 (t)  \,+ \,{\,1- \zeta_2 \, \over 2} \,  \mathrm{d} L^{X_1 - X_2} (t) \,+ \,{\,1- \eta_2 \, \over 2} \,  \mathrm{d} L^{X_2 - X_1} (t)\,,
\end{equation}

\medskip
\noindent
 now driven by the planar Brownian motion $\, (V_1 (\cdot), V_2 (\cdot))\,$.

\subsection{Skew Representations}
\label{sec3.2}

\noindent
In light of the \textsc{Tanaka} formula  (\ref{Tanaka1}), of the equation (\ref{SBBBM}) for the semimartingale $\, Y(\cdot)\,$, and of the last intertwinement in (\ref{2.10}), we  represent now the size of the ``gap'' between $\, X_1 (t)\,$ and $\, X_2 (t)\,$ as 
\begin{equation}
\label{2.13}
 \big| Y (t) \big|   \,=\,| y | - \lambda \,t   +  V^\flat (t) + 2\, \widehat{L}^{\,Y} (t)
 ~~~~~~~~~
\end{equation}
$$
~~~~~~~~~~~~~~~~~~~~~~~~~~~~~~~=\,| y | - \lambda \,t   +  V^\flat (t) +   \widehat{L}^{\,|Y|} (t)\,, \qquad 0 \le t < \infty  \,.
$$

\medskip
\noindent
With the help of (\ref{2.10}), (\ref{alpha}), let us write the first Brownian motion in (\ref{2.9.b}) as
$$
W^\flat (\cdot) \,=\, \gamma\, W (\cdot) + \delta\, U^\flat (\cdot)\,, \quad ~~\mathrm{where} \quad ~\gamma := \rho^2 - \sigma^2\,, ~~ ~~\delta := \sqrt{1-\gamma^2\,} = 2 \, \rho \, \sigma\,.
$$
With this notation, and with the help of (\ref{2.10.a}), the Brownian motion $\, V(\cdot)\,$ in (\ref{2.7}) takes the form
$$
V (t) \,=\, \gamma \,    V^\flat  (t) +   \delta\, Q (t) \,=\, \gamma \, \big( |Y(t)| - |y| + \lambda \, t - 2 \, \widehat{L}^Y (t) \big) + \delta \, Q(t)\,, \qquad 0 \le t < \infty\,.
$$
We recall here from (\ref{2.10.a}) the standard Brownian motion $\, Q(\cdot)\,$ which, being independent of $\, W(\cdot)\,$, is also independent of the process $\, Y(\cdot)\,$ in light of (\ref{YW}). 

\smallskip
In conjunction with $\, X_1(t) - X_2 (t) = Y(t)\,$ and the representation 
(\ref{2.12.a}) for $\, X_1(t)+ X_2 (t)  $, and with the notation
$$
\mu\,:=\, { 1 \over \,2\,}\, \big( \nu + \lambda \, \gamma \big)\,=\, g \, \rho^2 - h\, \sigma^2\,,
$$
we obtain from this expression the   {\it skew representations} for the component processes themselves
\begin{equation}
\label{2.23b}
X_1 (t) = x_1 + \mu\, t + \rho^2  \big(  Y^+ (t) - y^+ \big) - \sigma^2 \big( Y^- (t) - y^- \big) + \big(   1    -   \beta - \gamma \big)\, \widehat{L}^{\,Y} (t) + \rho  \sigma  Q (t) ~
\end{equation}
\begin{equation}
\label{2.24}
X_2 (t) = x_2 + \mu\, t - \sigma^2 \,\big(  Y^+ (t) - y^+ \big) + \rho^2  \big(  Y^- (t) - y^- \big) + \big(   1   -   \beta - \gamma \big)\, \widehat{L}^{\,Y} (t)   +    \rho  \sigma  Q (t) ~  
\end{equation}

 \medskip
\noindent
in terms of the paths of the skew Brownian motion process $\, Y(\cdot)\,$ with bang-bang drift, and of the {\it independent} Brownian motion $\, Q(\cdot)\,$. In particular, this shows that {\it uniqueness in distribution} holds for the system of stochastic differential equations (\ref{1.3}), (\ref{1.4}).

Similar reasoning shows  that   uniqueness in distribution   holds  also   for each of the systems (\ref{1.3.a}), (\ref{1.4.a}) and (\ref{1.3.b}), (\ref{1.4.b}). 

\begin{rem}
 \label{RFLSBBBM}
It is clear from (\ref{2.13}) that the absolute value of the skew Brownian motion with bang-bang drift in (\ref{SBBBM}), for {\it any} value $\, \alpha \in [0,1]\,$ of the skewness parameter, is Brownian motion with drift $\, - \lambda\,$ and reflection at the origin. Arguing as in \textsc{Walsh} \cite{TempsLocaux78} 
, Proposition 1, one can conclude that every  diffusion process $\, Y(\cdot)\,$,  for which $\, | Y(\cdot)|\,$ is Brownian motion with drift $\, - \lambda\,$  and  reflected at the origin, is a skew Brownian motion with bang-bang drift.
\end{rem}


\section{ Synthesis}
\label{sec4}

\noindent
We  reverse now the steps of the   analysis in section \ref{sec3}. Let us start
with a filtered probability space $\, ( \Om,   \mathfrak{F}, \p)\,$, $
\mathbf{F} = \{ \mathfrak{F} (t) \}_{0 \le t < \infty }\,$ and with two independent, standard Brownian motion $\,  W_1  (\cdot) \,$, 
$\,  W_2  (\cdot) \,$ on it;   we shall assume 
$\, \mathbf{F} \equiv \mathbf{F}^{\, (W_1, W_2)}\,$, i.e., that the $\, \mathbf{F}  \,$ is the smallest filtration  satisfying the usual conditions, to which   the planar Brownian motion $\, ( W_1  (\cdot) , \,   W_2  (\cdot) )\,$ is adapted.  

With     given real constants $\, \zeta_1\,,\, \zeta_2\,,\, \eta_1\,,\, \eta_2 \,$ and  nonnegative constants $\, g  \,$, $\, h  \,$, $\, \rho  \,$,
$\, \sigma \,$ that satisfy (\ref{1.1}) and   (\ref{1.5}), with  a given
vector $\, (x_1, x_2 ) \in \R^2\,$, and with the notation of (\ref{2.1}), we
construct the pairs of independent Brownian motions  
\begin{align}
  \label{3.5.a}
  W(\cdot) \,:&=\, \rho\, W_1 (\cdot) \,+\, \sigma \, W_2 (\cdot)\,,&
  U^{\, \flat}(\cdot) \,:&=\, \sigma \, W_1 (\cdot) \,-\, \rho  \, W_2 (\cdot) \\
  \intertext{and} 
  \label{3.5.ab}
  U(\cdot) \,:&=\, \sigma \, W_1 (\cdot) \,+\, \rho \, W_2 (\cdot)\,,&
  W^{\, \flat}(\cdot) \,:&=\, \rho\, W_1 (\cdot) \,-\, \sigma \, W_2 (\cdot)  
\end{align}
as in   (\ref{2.9.b}), (\ref{2.7}). Clearly, $\,   \mathbf{F}^{\, (W_1, W_2)} \equiv \mathbf{F}^{\, (W , U^{\flat})} \equiv \mathbf{F}^{\, (U, W^{  \flat})}\,$.  

\smallskip
We construct also the pathwise unique, strong solution $\, Y(\cdot)\,$ of the
stochastic equation (\ref{SBBBM}) driven by the Brownian motion $\, W(\cdot)\,$ introduced in (\ref{3.5.a}).    With the process $\, Y(\cdot)\,$ thus in place, we introduce by analogy with (\ref{2.10}) the  independent Brownian motions
\begin{equation}
  \label{3.5}
  V_1(\cdot) \,=\, \int_0^{\,\cdot} \sgn \big( Y (t) \big) \, \mathrm{d} W_1 (t)\,,  \qquad  
  V_2(t) \,=\, -\int_0^{\,\cdot}  \sgn \big( Y (t) \big) \, \mathrm{d} W_2 (t) \,,
\end{equation}
and by analogy with (\ref{2.7}), (\ref{2.9.b}) the two additional pairs  of independent Brownian motions   
\begin{align}
  \label{3.5.b}
  V(\cdot) \,:&=\, \rho\, V_1 (\cdot) \,+\, \sigma \, V_2 (\cdot)\,, &
  Q^{\, \flat}(\cdot) \,:&=\, \sigma \, V_1 (\cdot) \,-\, \rho  \, V_2 (\cdot) 
  \intertext{and} 
  \label{3.5.bb}
  Q(\cdot) \,:&=\, \sigma \, V_1 (\cdot) \,+\, \rho \, V_2 (\cdot)\,, &
  V^{\, \flat}(\cdot) \,:&=\, \rho\, V_1 (\cdot) \,-\, \sigma \, V_2 (\cdot) \,.
\end{align}

\medskip
\noindent
$\bullet~$ 
We introduce also the continuous martingales
\begin{equation}
  \label{3.8}
  M_1 (\cdot) \,:=\, 
  \int_0^{\,\cdot}  \Big( \, \rho \, \mathbf{ 1}_{ \{ Y (t) > 0 \} } \, \mathrm{d} W_1
  (t)  + \sigma \,  \mathbf{ 1}_{ \{ Y (t) \le 0 \} }  \, \mathrm{d} W_2 (t)
  \Big)  
\end{equation}
$$
~~~~~~~~~~~=\, 
  \int_0^{\,\cdot}  \Big( \, \rho \, \mathbf{ 1}_{ \{ Y (t) > 0 \} } \, \mathrm{d} V_1
  (t)  + \sigma \,  \mathbf{ 1}_{ \{ Y (t) \le 0 \} }  \, \mathrm{d} V_2 (t)
  \Big)\,,
$$
\begin{equation}
  \label{3.9}
~~~~~~~  M_2 (\cdot) \,:=\, -
  \int_0^{\,\cdot}  \Big(   \rho \, \mathbf{ 1}_{ \{ Y (t) \le  0 \} } \, \mathrm{d} W_1 (t)  + \sigma \,  \mathbf{ 1}_{ \{ Y (t) > 0 \} }  \, \mathrm{d} W_2 (t) \Big)  
\end{equation}
$$
~~~~~~~~~~~~=\, \int_0^{\,\cdot}  \Big(   \rho \, \mathbf{ 1}_{ \{ Y (t) \le  0 \} } \, \mathrm{d} V_1 (t)  + \sigma \,  \mathbf{ 1}_{ \{ Y (t) > 0 \} }  \, \mathrm{d} V_2 (t) \Big)\,,
$$

\smallskip
\noindent
with $ \, \langle M_1, M_2 \rangle$ $ ( \cdot) \equiv   0\,$ and quadratic
variations 
$$
\langle M_1  \rangle (\cdot) = \int_0^{\,\cdot} \big( \, \rho^2 \, \mathbf{ 1}_{ \{ Y (t) > 0 \} }    + \sigma^2 \,  \mathbf{ 1}_{ \{ Y (t) \le 0 \} }  \big) \mathrm{d} t\,, \quad~~ \langle M_2  \rangle (\cdot)= \int_0^{\,\cdot} \big( \, \rho^2 \, \mathbf{ 1}_{ \{ Y (t) \le 0 \} }    + \sigma^2 \,  \mathbf{ 1}_{ \{ Y (t) > 0 \} }  \big) \mathrm{d} t\,.
$$
There exist then independent Brownian motions $\, B_1 (\cdot)\,$, $  B_2
(\cdot)\,$ on    our filtered probability space $\, ( \Om,   \mathfrak{F},
\p)\,$, $ \mathbf{F} = \{ \mathfrak{F} (t) \}_{0 \le t < \infty }\,$, so   the
continuous martingales of (\ref{3.8}), (\ref{3.9}) are cast in their \textsc{Doob} representations as  
\begin{equation}
  \label{3.11} 
  M_1 (\cdot) = 
  \int_0^{\,\cdot}    \Big(   \rho \, \mathbf{ 1}_{ \{ Y (t) > 0 \} }   + \sigma \,  \mathbf{ 1}_{ \{ Y (t) \le 0 \} }   \Big)\,   \mathrm{d} B_1 (t)\,,  \quad M_2 (\cdot) = 
  \int_0^{\,\cdot}  \Big(   \rho \, \mathbf{ 1}_{ \{ Y (t) \le  0 \} }   + \sigma \,  \mathbf{ 1}_{ \{ Y (t) > 0 \} }   \Big)\,   \mathrm{d} B_2 (t) 
\end{equation}
in terms of independent Brownian motions $\, B_1 (\cdot)\,$, $\, B_2 (\cdot)\,$; for instance, by taking   
\begin{align}
  \label{3.12} 
  B_1 (\cdot) \, &=  \,\hphantom{-}\int_0^{\,\cdot}  \Big(  \, \mathbf{ 1}_{ \{ Y (t) > 0 \} } \,
  \mathrm{d} W_1 (t)  +  \mathbf{ 1}_{ \{ Y (t) \le 0 \} }  \, \mathrm{d} W_2
  (t) \Big)\,, \\
  \label{3.13} 
  B_2 (\cdot)  \, & =    \, -\int_0^{\,\cdot}  \Big(  \mathbf{ 1}_{ \{ Y (s) \le 0 \} } \,
  \mathrm{d} W_1 (t)  +  \mathbf{ 1}_{ \{ Y (t) > 0 \} }  \, \mathrm{d} W_2
  (t) \Big)   \,.
\end{align}

\medskip
\noindent
$\bullet~$ 
Finally, we introduce   the continuous, $\, \mathbf{F}-$adapted processes 
\begin{equation}
  \label{3.6}
X_1(\cdot) :=   x_1 + 
  \int_0^{\,\cdot}  \left( \,g\, \mathbf{ 1}_{ \{ Y (t) \le 0 \} } - 
    h\,\mathbf{ 1}_{ \{ Y (t) > 0 \} } \right)   \mathrm{d} t\, +\,M_1 (\cdot)
    \,+\, {\,1- \zeta_1 \, \over 2} \,   L^{Y} (\cdot) \,+ \,{\,1- \eta_1 \, \over 2} \,    L^{Y}_- (\cdot)   
    \end{equation}
    and
 \begin{equation}
  \label{3.7}
 X_2(\cdot) :=x_2 +    \int_0^{\,\cdot}  \left( \,g\, \mathbf{ 1}_{ \{ Y (t) > 0 \} } - h\,\mathbf{ 1}_{ \{ Y (t) \le 0 \} } \right)   \mathrm{d} t \,+\, M_2 (\cdot)  \, +  {\,1- \zeta_2 \, \over 2} \,   L^{Y} (\cdot) +{\,1- \eta_2 \, \over 2} \,    L^{Y}_- (\cdot)\,.
\end{equation}

\medskip
\noindent
It is now easy to check $\, X_1 ( \cdot) - X_2 ( \cdot)= Y ( \cdot)\,$; and from this,   that   the vector process $\, (X_1 ( \cdot) ,\, X_2 ( \cdot) )\,$ solves the system (\ref{1.3})-(\ref{1.4}), as well as the systems (\ref{1.3.a})-(\ref{1.4.a}), (\ref{1.3.b})-(\ref{1.4.b}). It is also straightforward to verify the skew representations of (\ref{2.23b}), (\ref{2.24}).

\begin{rem} 
 \label{aux}
We note that the vector process $\, (X_1 ( \cdot) ,\, X_2 ( \cdot) )\,$ solves also the system of stochastic equations 
\begin{equation}
\label{1.3.c}
 \mathrm{d} X_1 (t)= \mathbf{ 1}_{ \{ X_1 (t) \le  X_2 (t) \} }   \big(\, g \,  \mathrm{d} t + \sigma\,        \mathrm{d} B_1 (t) \big)~~~~~~~~~~~~~~~~~~~~~~~~~~~~~~~~~~~~~~~~
 \end{equation}
 $$
~~~~~~~~~~~~~~~~~~~~~~~~~~~~~ \,+ \,\mathbf{ 1}_{ \{ X_1 (t) >  X_2 (t) \} }  \big( - h \,    \mathrm{d} t + \rho  \,    \mathrm{d} B_1 (t) \big)+     \kappa_1  \, \mathrm{d} L^{|X_1 - X_2|} (t)\,,
 $$
 
 \begin{equation}
\label{1.4.c}
 \mathrm{d} X_2 (t)= \mathbf{ 1}_{ \{ X_1 (t) >  X_2 (t) \} }   \big(\, g \,  \mathrm{d} t + \sigma\,        \mathrm{d} B_2 (t) \big)~~~~~~~~~~~~~~~~~~~~~~~~~~~~~~~~~~~~~~~~
 \end{equation}
 $$
 ~~~~~~~~~~~~~~~~~~~~~~~~~~~~~ \,+ \,\mathbf{ 1}_{ \{ X_1 (t) \le  X_2 (t) \} }  \big( - h \,    \mathrm{d} t + \rho  \,    \mathrm{d} B_2 (t) \big)+     \kappa_2  \, \mathrm{d} L^{|X_1 - X_2|} (t)\,,
$$

\medskip
\noindent
where $\, 2 \,\kappa_j :=  \alpha \, ( 1-\zeta_j) + (1-\alpha) \,(1-\eta_j)\,$, $\, \,j=1, 2\,$ or equivalently
 \begin{equation}
\label{kappa}
\kappa_1\,=\, \alpha - \big( \beta / 2 \big)\,, \qquad \kappa_2\,=\, 1-\alpha - \big( \beta / 2 \big)\,.
 \end{equation}
\end{rem}

\subsection{Ranks}
\label{sec3.1}

\noindent
Let us   introduce   explicitly the ranked versions (leader and laggard, respectively)
\begin{equation}
\label{ranks}
R_1 (\cdot) \, := \, X_1 (\cdot) \vee X_2 (\cdot) \, , \qquad 
 R_2 (\cdot)\,:=\, X_1 (\cdot) \wedge X_2 (\cdot) 
\end{equation}
 of the components of the vector  process $\,   (X_1 (\cdot), X_2 (\cdot))\,$ constructed in (\ref{3.6}), (\ref{3.7}). From (\ref{2.12}) and (\ref{2.13}), it is rather clear that  we have  
 $$
  R_1 (t) + R_2 (t)\,=  \,X_1 (t) + X_2 (t) \,=\, r_1 + r_2 + \nu \,t +  V (t) +    \big(1-   \beta  \big) \,  L^{\,|Y|} (t )  \,, \quad 0 \le t < \infty
  $$
\begin{equation}
\label{2.15.a}  
   R_1 (t) - R_2 (t)\,=  \,\big| X_1 (t) - X_2 (t) \big|\,=  \,\big| Y (t) \big|  \,=\,
  | y | - \lambda \,t   +  V^\flat (t) +    L^{\,|Y|} (t)\,, 
   \end{equation} 
   
   \medskip
   \noindent
 and these representations lead to the  expressions 
   \begin{align}
   \label{2.16}
   R_1 (t)\,&=\, r_1 - h\, t + \rho \, V_1(t) + \big(1-  (\beta / 2) \big)\, L^{\,R_1 - R_2} (t)\,,   
   &&0 \le t < \infty\\
      \label{2.17}
   \,
   R_2 (t)\,&=\, r_2 + g\, t + \sigma \, V_2(t)- ( \beta / 2)\, L^{\,R_1 - R_2} (t)  \,,  
   &&0 \le t < \infty \,.
 \end{align}

 \medskip
 A few remarks are in order. The  equations (\ref{2.16}), (\ref{2.17})  identify the processes $\, V_1(\cdot)\,$ and $\, V_2(\cdot)\,$ of   (\ref{2.5}), (\ref{2.6})  as  the independent Brownian motions  associated with  the diffusive motion of the {\it ranked particles}, the ``leader''   $\, R_1 (\cdot)\,$ and the ``laggard''  $\, R_2 (\cdot)\,$, respectively; whereas the independent Brownian motions  $\, B_1(\cdot)\,$ in (\ref{1.3}) and $\, B_2(\cdot)\,$ in (\ref{1.4})  are associated with the  specific    ``names'' (indices, or identities) of the individual particles. On the other hand, with the help of the theory of the \textsc{Skorokhod} reflection problem (e.g., \textsc{Karatzas \& Shreve} \cite{MR1121940} 
 , page 210), we obtain from (\ref{2.15.a}), (\ref{flat})  the identification of the ``collision local time''
\begin{equation}
\label{2.14}
  L^{\,R_1 - R_2} (t)\, =\, L^{\,|Y|} (t)\, =\, \max_{0 \le s \le t} \Big( -  |y|  +  \lambda \, s     -  V^{\,\flat} (s)  \Big)^+\,, \qquad 0 \le t < \infty\,.
\end{equation} 
Let us also observe that, in the non-degenerate case  $\, \rho\,   \sigma >0 \,$, the equations (\ref{2.15.a})-(\ref{2.14}) and the second equation in (\ref{3.5.bb}) give the filtration comparisons
\begin{equation}
\label{flit}
\mathfrak{F}^{\,(V_1, V_2)} (t) \,  = \, 
\mathfrak{F}^{\,(R_1, R_2)} (t)  \,, \qquad 0 \le t < \infty
\,, 
\end{equation}
\begin{equation}
\label{flit1}
\mathfrak{F}^{\, V^\flat} (t) \,=\, \mathfrak{F}^{\, |Y|} (t)
\,  \subsetneqq  \, 
\mathfrak{F}^{\,Y} (t)  \,, \qquad 0 < t < \infty
\,,
\end{equation}
where the   inclusion is strict, due to the fact that the process $\, Y(\cdot)\,$ changes its sign with positive probability during any time-interval $\, [0,t]\,$ with $\, t >0\,$.

\subsection{Filtration Comparisons, Weak and Strong Solutions}
\label{sec4.2}

\noindent
We have the following straightforward analogues of Propositions 4.1, 4.2 and of Theorems 4.1, 4.2  in \textsc{Fernholz et al.} \cite{FIKP2012}.  

 \begin{prop}
  \label{Proposition 1}
  In the degenerate case   $\,   \sigma =0\,$, thus $\, \rho   =1\,$ in  light of (\ref{1.1}), we have   the relations 
  \begin{equation}
    \label{4.6}
    \mathfrak{F}^{\,(R_1, R_2)} (t)=\mathfrak{F}^{\,V} (t)  = \mathfrak{F}^{\,|X_1 - X_2|} (t)  \,   \subsetneqq  \, \mathfrak{F}^{\,X_1 - X_2} (t)=\mathfrak{F}^{\,W} (t)  =\mathfrak{F}^{\,(X_1, X_2)} (t)   
  \end{equation}
  for every $\, 0 < t < \infty\,$, where the inclusion is strict. 
  
  \smallskip
  In the special case $\, \beta =1\,$ of (\ref{B1}) we have in addition $\, \sigmab( V (t)) =\sigmab( X_1 (t)+ X_2 (t))\,$, thus   also  
   $\, \mathfrak{F}^{\,V} (t)  = \mathfrak{F}^{\,X_1 + X_2} (t)\,,\,$ for every $\, 0 \le t < \infty\,$. 
\end{prop}

\begin{prop}
  \label{Proposition 2}
  In the non-degenerate case $\, \rho \, \sigma >0\,$, we have for every $\,  0 < t < \infty\,$   the filtration relations 
  \begin{equation}
    \label{4.11}
    \mathfrak{F}^{\,(V_1, V_2)} (t) \,=\,\mathfrak{F}^{\,(R_1, R_2)} (t)\,=\, \mathfrak{F}^{\,(|Y|, V)} (t) \,=\, \mathfrak{F}^{\,(|Y|, Q)} (t)    ~~~~~~~~~~~~~~~~~~~~~~~~~~~~~~~~~~~~~~~~~~~~~~~~~~~~~ 
  \end{equation}
  $$  ~~~~~~~~~~~~~~~~~~~~~~~~~~~~~~~ ~~~~~~~~~~~~~~~~~~~~~~~~~~~~~~~ ~~~\subsetneqq  \mathfrak{F}^{\,(Y, Q)} (t)\,=\, \mathfrak{F}^{\,(Y, V)} (t)\,=\,\mathfrak{F}^{\,(W_1, W_2)} (t)\,=\,\mathfrak{F}^{\,(X_1, X_2)}(t)  \, ,
  $$
  where the inclusion is strict.
\end{prop}

\begin{thm}
  \label{Theorem 1} 
  The system of stochastic differential equations (\ref{1.3}), (\ref{1.4}) is well-posed, that is, has a weak solution which is unique in the  sense of the probability distribution. The same is true for each of the  systems of equations (\ref{1.3.a}), (\ref{1.4.a}) and (\ref{1.3.b}), (\ref{1.4.b}). 
  
  On the other hand,   the system of stochastic differential equations
  (\ref{1.3.a}), (\ref{1.4.a})  admits a     strong solution,  which is therefore pathwise unique;     whereas the system   (\ref{1.3.b}), (\ref{1.4.b})  admits no strong    solution. 
\end{thm}

\begin{thm}
  \label{Theorem 2} 
  The system of stochastic differential equations (\ref{1.3}), (\ref{1.4})
  admits a pathwise unique, strong solution; in particular, the filtration identity $\, \mathfrak{F}^{\,(B_1, B_2)} (t)\,=\,\mathfrak{F}^{\,(X_1, X_2)}(t) \,$ 
  holds for all $ \, 0 \le t < \infty\,$.    
  
  Likewise, the system of equations (\ref{1.3.c}), (\ref{1.4.c})  has a pathwise unique, strong solution. 
\end{thm}

\begin{proof}  Repeating almost verbatim the arguments in the proof of Theorem 5.1 in \textsc{Fernholz et al.} \cite{FIKP2012},  
the question boils down to whether the filtration comparison 
\begin{equation}
\label{FiltComp}
\mathfrak{F}^{\,Y} (t)\,\subseteq\,\mathfrak{F}^{\,(B_1, B_2)}(t) \,, \qquad \forall ~~~\, 0 \le t < \infty 
\end{equation}
holds. To decide this issue, we write the equation (\ref{SBBBM}) as driven by the pair $\,(B_1(\cdot), B_2 (\cdot))\,$; in other words, we use (\ref{2.7}) and (\ref{2.3}), (\ref{2.4}) to  express the skew Brownian motion $\, Y(\cdot)\,$  with bang-bang drift    as solution of a stochastic differential equation driven by the planar Brownian motion 
$\,(B_1(\cdot), B_2 (\cdot))\,$. Since this equation does admit a weak solution which is unique in distribution, the issue   is whether this solution is also strong, that is, whether (\ref{FiltComp}) holds.

This question is easy to settle in the isotropic case $\, \rho = \sigma = 1 / \sqrt{2\,}\,$; then $\, W (\cdot) = (B_1(\cdot)- B_2 (\cdot)) / \sqrt{2\,}\,$, and the comparison (\ref{FiltComp}) follows from the strong solvability of the equation (\ref{SBBBM}) proved in section \ref{sec5}: $\, \mathfrak{F}^{\,Y} (t) = \mathfrak{F}^{\,W} (t) \subseteq \mathfrak{F}^{\,(B_1, B_2)}(t) \,$ holds for all  $\, 0 \le t < \infty\,$, by virtue of (\ref{YW}). 

\smallskip
In the non-isotropic case $\, \rho \neq \sigma  \,$, we write (\ref{SBBBM}) as   the {\it extended skew Tanaka equation}
\begin{equation}
\label{ESTE}
Y ( t)\,=\, y \,+ \,{ \, \rho - \sigma \, \over \sqrt{2\,} } \int_0^{\,  t} \overline{\text{sgn}} \big( Y(s) \big)\, \mathrm{d} \betab (s) - { \, \rho + \sigma \, \over \sqrt{2\,} }\, \thetab ( t)\,+\, 2  \big(   2\, \alpha -1  \big)\, \widehat{L}^{\,Y} ( t)\,,\quad 0 \le t \le T
\end{equation}
with $\, T \in (0, \infty)\,$ arbitrary but fixed. Here $$\, \betab (\cdot)\,:=\, {\,\betab_1 (\cdot) + \betab_2 (\cdot)\, \over \sqrt{2\,\,}}\,, \qquad  \thetab (\cdot) \,:=\, {\,\betab_1 (\cdot) - \betab_2 (\cdot)\, \over \sqrt{2\,\,}}\,$$ are  given standard, independent  Brownian motions after an equivalent change of probability measure, and 
 $$
\betab_i ( t) \,:=\, B_i (t) - { \lambda \, t \over \, \rho - \sigma\,}\,, \qquad 0 \le t \le T ~~~~~~~~~~~(i=1, 2)\,. 
 $$
 
\medskip
\noindent
 $\bullet~$ Let us suppose that  $\,Y_{1}(\cdot)\,$ and $\,Y_{2}(\cdot)\,$ are   two solutions of  the equation (\ref{ESTE}), defined on the same probability space and with respect to the same, independent standard Brownian motions $\, B_1(\cdot), \, B_2 (\cdot) \,$. Following \textsc{Le Gall} (1983), we shall show $\, {L}^{Y_{1} - Y_{2}} (\cdot) \equiv 0\,$; we shall then argue  that this implies also  $\,Y_{1}(\cdot) \equiv Y_{2}(\cdot)\,$, a.s. 

\smallskip
To this end, we consider   the difference  $\,D(\cdot) \, :=\, Y_{1}(\cdot) - Y_{2}(\cdot)\,$, as well as the linear combinations 
$$\, 
Z^{(u)}(\cdot) \, :=\, (1-u) \, Y_{1}(\cdot) + u\, Y_{2}(\cdot) \,~ ~~~~ \hbox{for} ~~~\,0 \le u \le 1\,;
$$ 
we introduce also a sequence  $\, \{ f_k \}_{k \in \N} \subset
  \mathcal{C}^1(\R)\,$  of continuous and continuously differentiable
  functions that converge to $\,f_{\infty} (\cdot) \, :=\, \overline{\text{sgn}}(\cdot)\,$ pointwise,   and
  satisfy $\sup_{k\in\N}\norm{f_k}_{TV}<\infty$. Since $\limsup_{k}
  \norm{f_k}_{TV}\leq \norm{f_{\infty}}_{TV}$ obviously holds, this is only possible if  $\,f_{\infty}(\cdot)\,$ is of bounded variation, and in this case an approximating sequence is  easily  obtained, e.g., by mollifiers.  
As in the proof of Theorem 8.1 of \textsc{Fernholz et al.} \cite{FIKP2012}, for every $\,\delta > 0\,$, $\,T > 0 \,$, $\,k \ge 1\,$, we establish then 
\[
\mathbb E \Big[ \int^{T}_{0} \frac{\, \lvert f_{k}(Y_{1}(s)) - f_{k}(Y_{2}(s)) \rvert\, }{Y_{1}(s) - Y_{2}(s)} {\bf 1}_{\{ Y_{1}(s) - Y_{2}(s) > \delta\}} \, \mathrm d t \,\Big] \, \le\,   c_{1} \, \lVert f_{k} \rVert_{TV} \cdot \sup_{\xi, u} \mathbb E \big( 2 \widehat{L}^{(u)} (T, \xi) \big) \,   ;
\]
 here $\, \widehat{L}^{(u)}(T, \xi)\,$ is the symmetric local time of $\, Z^{(u)}(\cdot)\,$ accumulated at the site $\,\xi \in \R\,$ over the time interval $\, [0, T]\,$, and $\,c_{1}\,$ is a constant chosen independently of $\, k, u, \delta \,$. Letting $\,k \uparrow \infty\,$ and $\,\delta \downarrow 0\,$, we obtain 
\[
\mathbb E \Big[ \int^{T}_{0} \frac{1}{\,D(t)\,} \,{\bf 1}_{\{D(t) > 0 \}}\, \mathrm d \langle D\rangle(t)\, \Big]  \, \le \, 
2\, \mathbb E \Big[ \int^{T}_{0} \frac{\, \lvert f_{\infty}(Y_{1}(t)) - f_{\infty}(Y_{2}(t)) \rvert\, }{Y_{1}(t) - Y_{2}(t)} {\bf 1}_{\{ D(t) > 0\}} \mathrm \,d t \,\Big]~
\]
\[
\hspace{2.2cm} ~~\,~~~~~~~\le \, 2 \,c_{1}\,  \lVert f_{\infty} \rVert_{TV} \cdot  \sup_{\xi, u} \mathbb E \Big( 2 \widehat{L}^{(u)} (T, \xi) \Big) \, .  
\]
Now the \textsc{Cauchy-Schwartz} inequality, the \textsc{It\^o}  isometry, and the \textsc{Tanaka}  formula (\ref{Tanaka1}) applied to $\, Z^{(u)}(\cdot)\,$, 
allow us to  estimate 
\[
\mathbb E \big(2 \widehat{L}^{(u)}(T; \xi) \big) \le \mathbb E \lvert Z^{(u)}(T) - Z^{(u)}(0) \rvert + \big[ \mathbb E ( \langle Z^{(u)}\rangle (T) ) \big]^{1/2} 
\]
\[
\hspace{4cm}~~~~~~~~ + \,2\,  ( 2\alpha -1 ) \big( u \, \mathbb E\big(\widehat{L}^{Y_{1}}(T) \big) + (1-u) \, \mathbb E \big( \widehat{L}^{Y_{2}}(T)\big) \big) \]
\[
\hspace{0.5cm} \le \,2\,  \Big[ \big [ \mathbb E ( \langle Z^{(u)}\rangle (T) ) \big ] ^{1/2} +2\,  ( 2\alpha -1 ) 
\Big( u \, \mathbb E\big(\widehat{L}^{Y_{1}}(T) \big) + (1-u
) \, \mathbb E \big( \widehat{L}^{Y_{2}}(T)\big) \Big) \Big] \, . 
\]

\smallskip
\noindent
The last term is bounded uniformly in $\, (\xi, u)\,$, since $\, \langle Z^{(u)}\rangle (t) \le c_{2} \, t\, $  and $\,\mathbb E \big(\widehat{L}^{\,Y_{i}}(T)\big) \le c_{3}\,$, for $\,i = 1, 2\,$ and for some constants $\, c_{2}\,, c_{3}\,$ that do not depend on $\,(\xi, u)\,$. Thus, we obtain 
\begin{equation} 
\label{eq: Lemma1 LeGall}
\mathbb E \Big[ \int^{T}_{0} \frac{1}{\,D(t)\,} \,{\bf 1}_{\{D(t) > 0 \}}\, \mathrm d \langle D\rangle(t)\, \Big] < \infty \, , \quad 0 < T < \infty \, . 
\end{equation}

\medskip
Using Lemma 1.0 of \textsc{Le Gall} \cite{MR0770393} 
(see also Exercise 3.7.12, pages 225-226 in \textsc{Karatzas \& Shreve} \cite{MR1121940}
), we verify that (\ref{eq: Lemma1 LeGall}) gives $\,L^{D}(\cdot) \equiv 0\,$. By exchanging the r\^oles of $\,Y_{1}(\cdot)\,$ and $\,Y_{2}(\cdot)\,$, we obtain also $\,L^{-D}(\cdot) = L^{Y_{2} - Y_{1}}(\cdot) \equiv 0\,$, as well as $\, \widehat{L}^{D}(\cdot) \equiv 0\,$. Furthermore, we note that on the strength of Corollary 2.6 of \textsc{Ouknine \& Rutkowski} \cite{MR1324099} 
 this implies that the symmetric local time $\, \widehat{L}^{M}(\cdot) \,$ of the maximum $$\,M(\cdot) \, :=\, Y_{1}(\cdot) \vee Y_{2}(\cdot) \, =\, Y_{1}(\cdot) + \big(Y_{2}(\cdot) - Y_{1}(\cdot) \big)^{+}\,$$ is   given as 
\[
\widehat{L}^{M}(\cdot) \, :=\,  \widehat{L}^{\,Y_{1} \vee Y_{2}}(\cdot) \, =\, \int_{0}^{\,\cdot} 
{\bf 1}_{\{Y_{2}(t) \le 0\}} \, {\mathrm d} \widehat{L}^{\,Y_{1}}(t) +  \int_{0}^{\,\cdot} {\bf 1}_{\{Y_{1}(t) < 0\}} \, {\mathrm d} \widehat{L}^{\,Y_{2}}(t)   \, . 
\]
We combine now these results with the \textsc{Tanaka}  formula, to obtain the dynamics of the maximum  
\[
M(\cdot) \, =\, y + \int_0^{\, \cdot} {\bf 1}_{\{Y_{1} (t) \ge Y_{2}(t)\}}\, {\mathrm d} Y_{1}(t)   + \int_0^{\, \cdot} {\bf 1}_{\{Y_{1}(t)< Y_{2}(t)\}} \, {\mathrm d} Y_{2} (t) +      
L^{Y_{2} - Y_{1}}(\cdot) ~~~~~~
\]
\[
~~~~ ~~~\, =\,y+  \frac{\, \rho - \sigma \, }{\sqrt{2\,\,}}   \int_0^{\, \cdot} \, \overline{\text{sgn}} (M(t)) \, {\mathrm d} {\bm \beta} (t)  - \frac{\, \rho + \sigma \, }{2} \, 
{\bm \vartheta} (\cdot) + 2\,\big(2\alpha - 1\big) \,  
\widehat{L}^{M}(\cdot)  \,,
\]

\medskip
\noindent
and observe that these are the same as those of (\ref{ESTE}). But uniqueness in distribution holds for the equation (\ref{ESTE}), so the distribution of the process $\,M(\cdot)\,$ is the same as that of $\,Y_{1}(\cdot)\,$; and of course we have $\, M (\cdot) \ge Y_{1}(\cdot)\,$ a.s. This implies $\, M (\cdot) \equiv  Y_{1}(\cdot)\,$, thus $\, Y_1 (\cdot) \equiv  Y_{2}(\cdot)\,$ a.s. 

\smallskip
Therefore, the solution  to (\ref{ESTE}) is pathwise unique,   hence also strong by the theory of \textsc{Yamada \& Watanabe} (e.g., \textsc{Karatzas \& Shreve} \cite{MR1121940}
, pages 308-311). 
 \end{proof}


\section{Some Special Cases}
\label{sec6}

\noindent
When $\, \alpha = 1/2\,$, that is, $\, \eta = \zeta \neq 0\,$ or equivalently \begin{equation}
\label{A1/2} 
\eta_1 - \eta_2\,=\, \zeta_2 - \zeta_1 \,\neq\,2\,,
\end{equation}
 the equation (\ref{SBBBM}) for the difference $\, Y(\cdot) = X_1 (\cdot) - X_2 (\cdot)\,$ becomes that of Brownian motion with bang-bang drift    
$$
Y (t) \,=\, y - \lambda \int_0^t  \mathrm{sgn}  \big( Y(s)\big)\,   \mathrm{d} s     + W(t)\,, \qquad 0 \le t < \infty 
$$
as in (\ref{KS}).  In this special case $\, \eta = \zeta \neq 0\,$ and with $\,\sigma = \rho\,$, the existence and uniqueness of (\ref{1.3})-(\ref{1.4}) can be shown also by   direct application of Theorem 3.5 of \textsc{Sznitman \& Varadhan} \cite{MR833269} 
and a \textsc{Girsanov}'s change-of-measure, with the aid of the local time relationships (\ref{alpha}).

On the other hand, when $\, \beta =1\,$ or equivalently
 \begin{equation}
\label{B1}
\eta \, \big( 1 - \overline{\zeta} \,\big) \,=\, \zeta \, \big( 1 - \overline{\eta} \, \big)\,,
\end{equation}
we observe from (\ref{2.12.a}) that the sum $\, X_1 (\cdot) + X_2 (\cdot)\,$ is just  standard Brownian motion with drift $\, \nu = g-h\,$. 

\smallskip
Let us single out now, and study, some more interesting special cases.

\subsection{Perfect Reflection for Individual Particles Upon Collision}
\label{Rem1}
 
 \noindent
Suppose  $\, \alpha =1\,$, or equivalently $\, \zeta =0\,$ and  $\, \eta \neq 0\,$ from (\ref{1.5}), that is 
\begin{equation}
\label{1.10}
\zeta_2 - \zeta_1\,=\, 2\,\neq\, \eta_1 - \eta_2 \,,
\end{equation}
and that $\, x_1 \ge x_2\,$. We see then from (\ref{alpha})    that we have $\, L^Y_- (\cdot) \equiv L^{X_2 - X_1} (\cdot) \equiv 0\,$,  and that the equation (\ref{SBBBM}) becomes $$ Y (t) \,=\, y - \lambda \, t + W(t) + L^Y (t)\,, \qquad 0 \le t < \infty\,.$$ From the theory of the \textsc{Skorokhod} problem (e.g., \textsc{Karatzas \& Shreve} \cite{MR1121940}
, pages 209-210) we conclude 
$$
L^Y ( t) \,=\, \max_{\,0 \le s \le  t} \big( -y + \lambda \, s - W(s)\big)^+\,, \qquad 0 \le t < \infty 
$$ 
thus $\, Y (\cdot) \ge 0\,$, and that strength and pathwise uniqueness hold; for more general results along these lines see \textsc{Chitashvili  \& Lazrieva} \cite{MR636083}. 
  It is also clear from the last two displayed equations, that   the filtration identity $ \,  \mathfrak{F}^Y (t) \,=\, \mathfrak{F}^W (t)\,, ~ \, \,0 \le t < \infty\,$ in (\ref{YW}) also holds.

In this case, then, when the particles collide,  the trajectory $\, X_1(\cdot)\,$ of the first particle  bounces off the trajectory $\, X_2(\cdot)\,$ of the second particle   as if this latter    were a perfectly reflecting lower boundary. We can visualize the situation by saying that, under the conditions of (\ref{1.5}) and (\ref{1.10}), the second particle is ``heavy'' (unaffected by collisions), whereas the first particle is ``light'' in that it bounces off (reflects perfectly) when colliding with the heavy particle. 

\medskip
\noindent
$\bullet~$ 
The ``symmetric'' situation obtains for $\, \alpha =0\,$,  that is $\, \zeta \neq 0\,$ and  $\, \eta = 0\,$  
or equivalently
\begin{equation}
\label{1.11}
\zeta_2 - \zeta_1\,\neq\, 2\,=\, \eta_1 - \eta_2 \,\,;
\end{equation}
in this case and again with $\, x_1 \ge x_2\,$, when the two particles collide,   the second particle bounces off the first as if this latter were a perfectly reflecting upper boundary; it is the first particle that is now ``heavy'', and the second that is ``light''.

\subsection{Frictionless Collision}
\label{Rem2}

\noindent
It follows also from (\ref{alpha}) that the local times disappear entirely in (\ref{1.3}) when we have the configuration of parameters $\, (1-\zeta_1) \alpha + (1-\eta_1)(1 - \alpha) =0\,$, or equivalently
\begin{equation}
\label{1.8}
(1-\zeta_1)\, \eta + (1-\eta_1)\, \zeta \,=\,0\,;
\end{equation}
in this case the trajectory of the first particle crosses that of the second without ``feeling it'', that is, without being subjected to any local time drag. 

Similarly, the  second particle crosses the first in the same frictionless manner, that is, the local times disappear entirely in (\ref{1.4}), if 
\begin{equation}
\label{1.9}
(1-\zeta_2)\, \eta + (1-\eta_2)\, \zeta \,=\,0\,.
\end{equation}
$\bullet~$  If both (\ref{1.8}) and (\ref{1.9}) hold, then all such crossings are {\it completely frictionless}. We note that (\ref{1.8}) and (\ref{1.9}) are both satisfied, if and only if 
\begin{equation}
\label{Frictionless}
\eta_1 + \zeta_1 \,=\,\eta_2 + \zeta_2 \,=\,2 
\end{equation}
holds. This condition implies $\, \eta = \zeta\,$ (so when this common value is nonzero we are in the   case $\, \alpha = 1/2\,$ mentioned at the start of the section), and  
  is obviously satisfied in the special case $\, \eta_1 = \zeta_1 = \eta_2 = \zeta_2 =1\,$   studied by \textsc{Fernhoz et al.} \cite{FIKP2012}.  
  However, (\ref{Frictionless}) holds  also   for other configurations of  parameters, for instance  $\,\zeta_1=  \eta_2 =1/2\,$, $\, \eta_1 =  \zeta_2 =  3/2\,$. 
  
  The condition (\ref{Frictionless}) gives   the value $\, \beta =1\,$ for the parameter of (\ref{SP1}); back in (\ref{2.16}), (\ref{2.17}), this implies that the collision local time $\, L^{R_1 - R_2}(\cdot)\,$ ``gets apportioned equally to the ranks''.

\subsection{Elastic Collisions}
\label{Elast}

\noindent
Beyond these two extremes  of perfect reflection and frictionless collision --  that is, for all other configurations of parameters -- we have collisions that are ``elastic'': neither completely frictionless, nor perfectly reflecting.

\subsection{Brownian motion reflected on an independent Brownian motion}
\label{Rem3}

\noindent
 Finally, let us consider the case $\, \beta   =0\,$ or equivalently 
$\,    \eta\, \overline{\zeta} + \zeta\, \overline{\eta}     =0\,$, that is
\begin{equation}
\label{1.12}
2\, \big(\zeta_1 + \zeta_2+\eta_1 + \eta_2\big)  \,=\,\big(\zeta_1 + \zeta_2 \big) \, \big( \eta_1 - \eta_2\big)  - \big( \eta_1 + \eta_2\big)\, \big(\zeta_1 - \zeta_2 \big)
\end{equation}
in light of (\ref{SP1}) and (\ref{1.6}), (\ref{1.7}). This happens, for instance, when $\, \zeta_1 = 3/4\,$, $\, \zeta_2 = 9/4\,$, $\, \eta_1 = - 4 / 3\,$, $\, \eta_2 = - 8/3\,$; in this case we have $\, \alpha = 4 /7\,$ and of course $\, \beta   =0\,$.  

\smallskip
Under the condition (\ref{1.12}),    the laggard in (\ref{2.17}) feels no pressure (local time drag) from the leader; it just evolves like Brownian motion with variance $\, \sigma^2\,$ and  nonnegative drift. On the other hand, the leader in (\ref{2.16}) evolves like an independent Brownian motion with variance $\, \rho^2\,$ and  nonpositive drift, reflected off the trajectory of the laggard. Such a process has been studied by \textsc{Burdzy   \& Nualart} \cite{MR1902187} 
(see also \textsc{Soucaliuc et al.} \cite{MR1785393}, 
   \textsc{Soucaliuc   \& Werner} \cite{MR1917545});  
here it arises as a special case of the ranked system  (\ref{2.16}), (\ref{2.17}) for the particles whose motions are governed by the equations (\ref{1.3}), (\ref{1.4}). 

We have in this  case  $\, \beta =  0 \,$  an interesting fusion: the ``perfect reflection'' we saw in subsection \ref{Rem1}, and   the  ``frictionless motion'' of subsection \ref{Rem2}, are occurring here simultaneously -- not for the  motions of the individual particles, however, but rather  for the motions of their  ranked versions, the leader $\,R_1(\cdot)\,$ and the laggard $\,R_2(\cdot)\,$, respectively. To put it a little differently: starting with two particles that undergo skew-elastic collisions one is able, under the conditions of (\ref{1.5}) and (\ref{1.12}), to ``simulate  a heavy particle'' (the laggard) and a ``light'' particle (the leader).

\medskip
\noindent
$\bullet~$ 
The ``reverse'' situation obtains when $\, \beta   =2\,$ or equivalently 
$\,    \eta\, \overline{\zeta} + \zeta\, \overline{\eta} \,=2\, (\eta  + \zeta)\,$, that is
\begin{equation}
\label{1.13}
2\, \big(\zeta_1 + \zeta_2+\eta_1 + \eta_2\big)  \,=\,4 \,\big(4+ \zeta_1 - \zeta_2-\eta_1 + \eta_2\big)+\,\big(\zeta_1 + \zeta_2 \big)  \big( \eta_1 - \eta_2\big)  - \big( \eta_1 + \eta_2\big)  \big(\zeta_1 - \zeta_2 \big)\,;~~~
\end{equation}
 then it is the trajectory of the laggard (now the ``light'' particle) that gets reflected off that of the leader (now the ``heavy'' particle). 
This happens, for instance, when $\, \zeta_{1}  = 3  / 2\, $, $\, \zeta_{2} = 3\,$, $\, \eta_{1} = 7 / 3\,$, $\, \eta_{2} = 1\,$; in this case we have $\, \alpha = 4 /7\,$ and $\, \beta   =2\,$.

\subsection{Some Simulations}
\label{Sim}

\noindent
The pictures (Figures \ref{f1}-\ref{f4}) that follow present simulations of the processes $\, X_1 (t)\,$ (in black) and $\, X_2 (t)\,$ (in red) for $\, t \in [0,1]\,$, in black and red, respectively,  with drifts $\, g=h=1\,$ in the degenerate case  $\, \rho =0\,$.

\vspace{-40pt}  
\begin{figure}[H]
\begin{center}
\scalebox{.61}{ \rotatebox{0}{\hspace{-40pt}\includegraphics{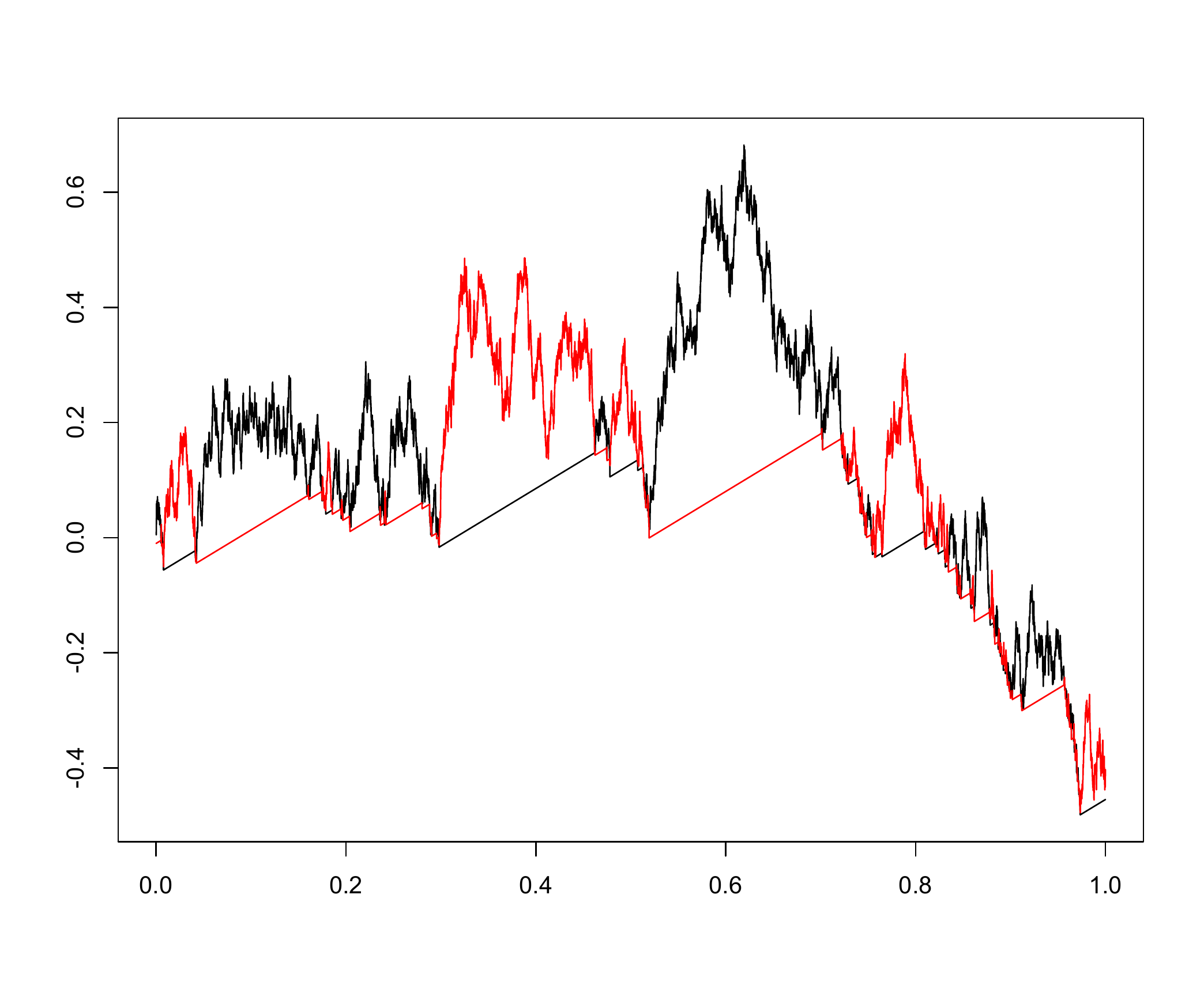}}}
\vspace{-30pt}
\caption{$~ \zeta_1 = \zeta_2 = \eta_1 = \eta_2 =1$
 ; $\, \alpha = 1/2 \, $, $\, \beta = 1\,$. 
}
\label{f1}
\end{center}
\vspace{-40pt}
\end{figure}

\begin{figure}[H]
\vspace*{-20pt}
\begin{center}
\scalebox{.61}{ \rotatebox{0}{
\hspace{-40pt}\includegraphics{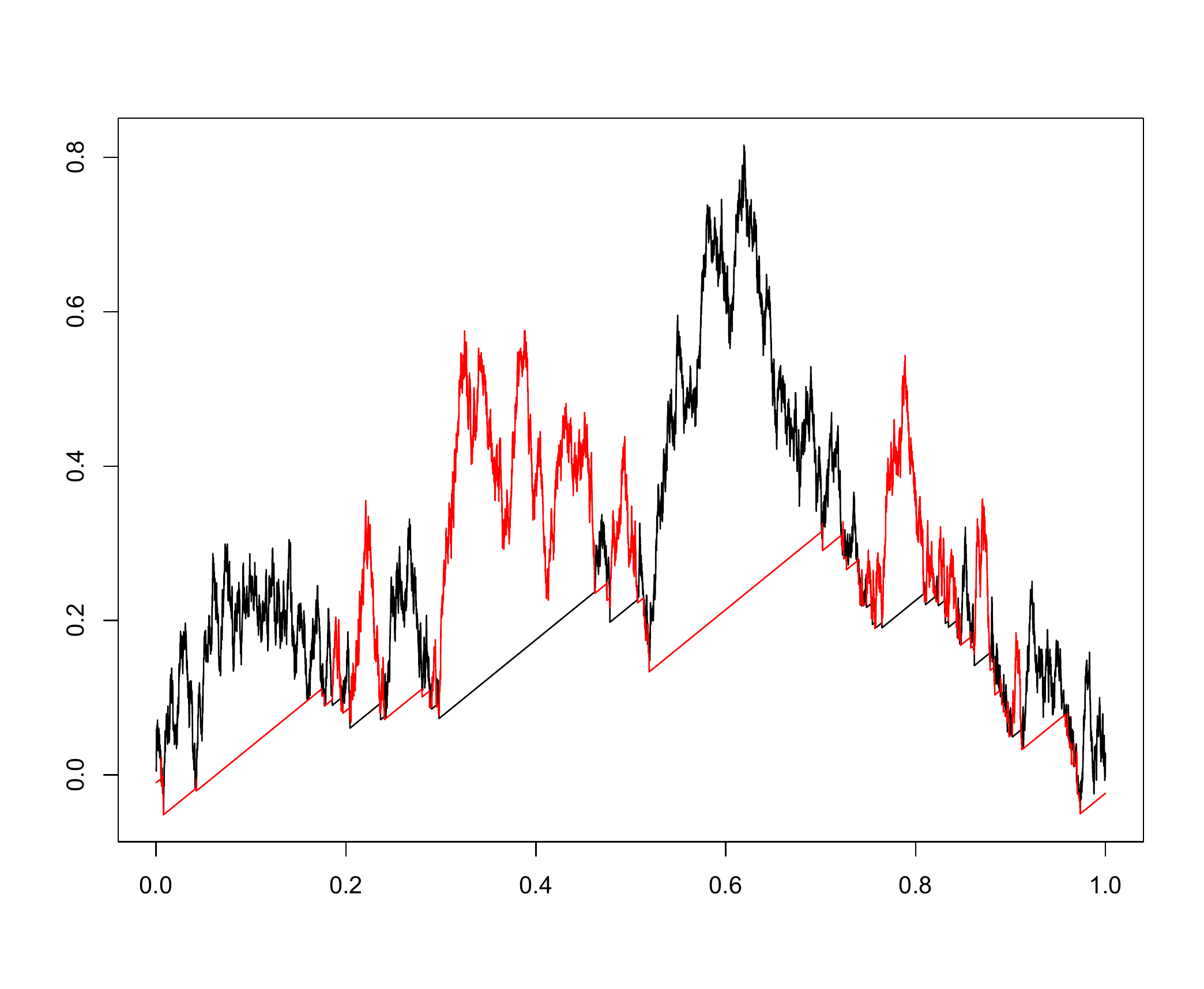}}}
\vspace{-30pt}
\caption{$~\zeta_1 =0\,$, $\, \zeta_2 = \eta_1 = \eta_2 =1$ 
 ; $\,\alpha = 2/3\,$, $\, \beta = 2/3\,$. 
}
\label{f2}
\end{center}
\vspace{-40pt}
\end{figure}

\begin{figure}[H]
\begin{center}
\scalebox{.61}{ \rotatebox{0}{
\hspace{-40pt}\includegraphics{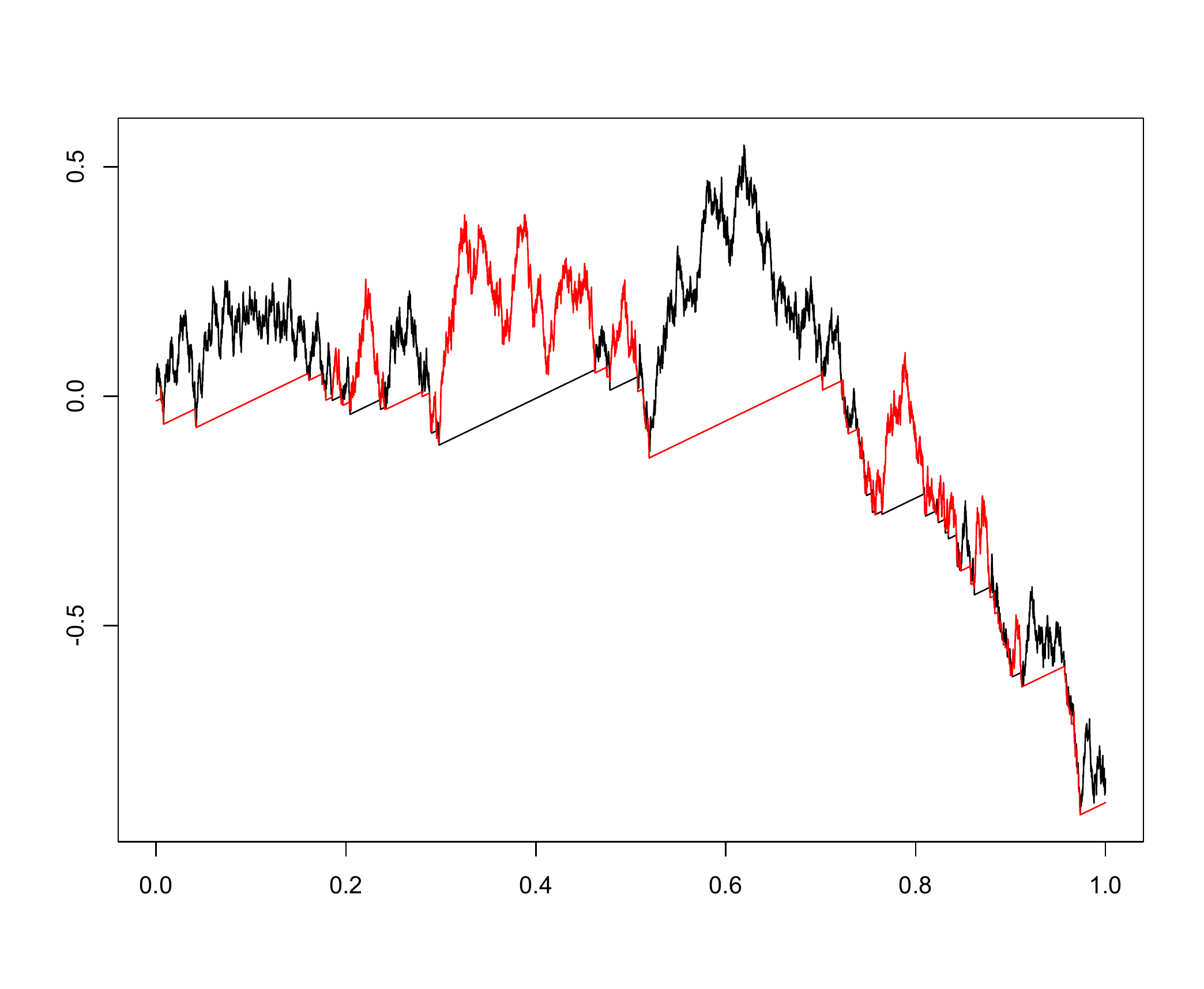}}}
\vspace{-30pt}
\caption{$~\zeta_2 =2\,$, $\, \zeta_1 = \eta_1 = \eta_2 =1$
 ; $\, \alpha = 2/3\,$, $\, \beta = 4/3\,$. 
}
\label{f3}
\end{center}
\vspace{-40pt}
\end{figure}

\begin{figure}[H]
\vspace*{-20pt}
\begin{center}
\scalebox{.61}{ \rotatebox{0}{
\hspace{-40pt}\includegraphics{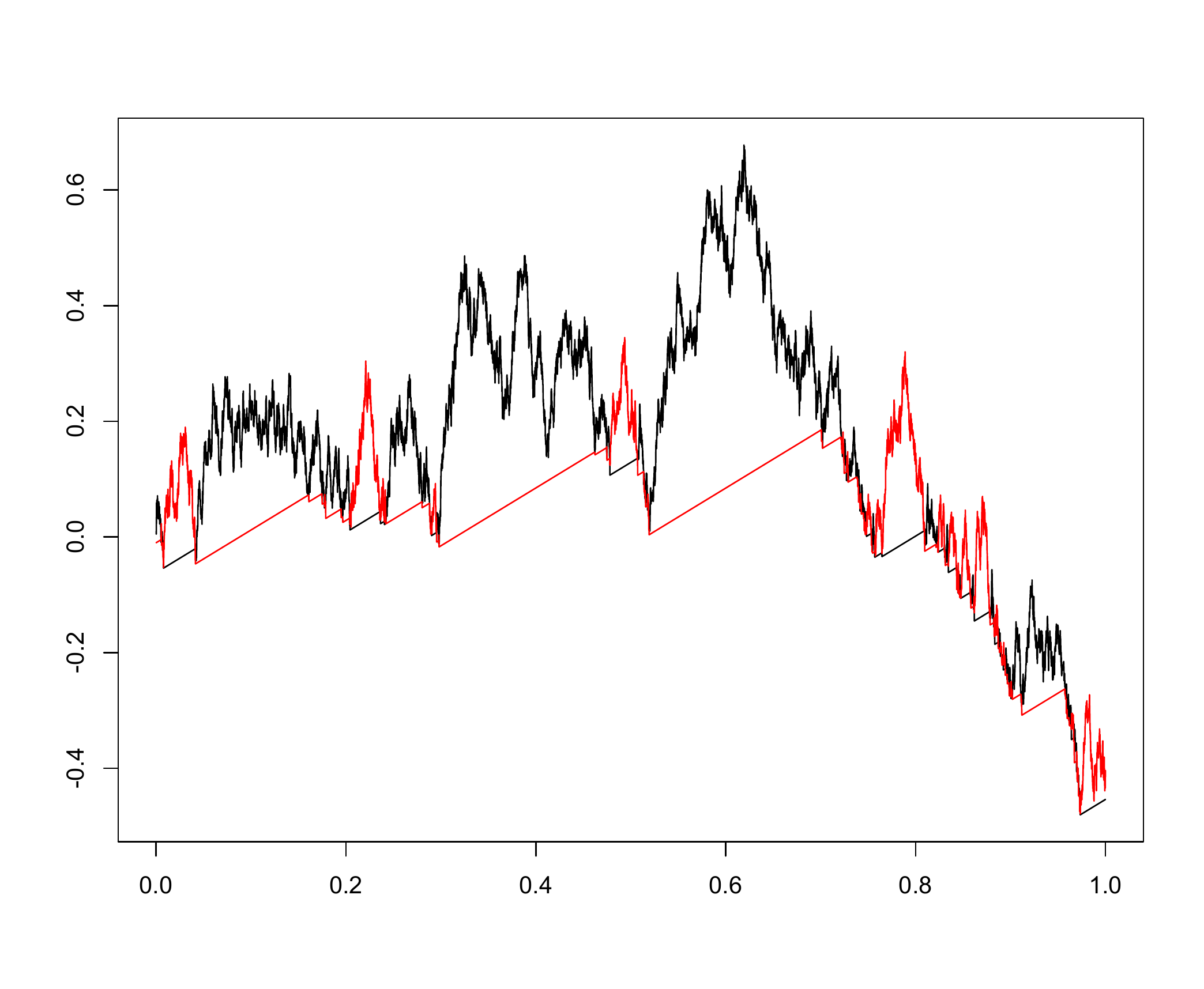}}}
\vspace{-30pt}
\caption{$~\zeta_1 =0\,$, $\, \zeta_2=2\,$, $\, \eta_1 = \eta_2 =1$ 
 ; $\, \alpha = 1\,$, $\,\beta = 1\,$. 
}
\label{f4}
\end{center}
\vspace{-40pt}
\end{figure}


\bigskip
\bigskip


\section{Skew Brownian Motion with Bang-Bang Drift }
\label{sec5}

\noindent
We   study here the stochastic differential equation (\ref{SBBBM}) for the skew Brownian motion with bang-bang drift  
 \begin{equation}
\label{BBD}
\mathfrak{ b} (y) \,=\, - \lambda \, \text{sgn} (y)\,,\qquad y \in \R
\end{equation}
for some given constant $\, \lambda >0\,$, with the notation of (\ref{sgn}), and with skewness parameter $\,    \alpha \in [0, 1]\,$. The cases $\, \alpha =0\,$ and $\, \alpha =1\,$ have been discussed already in subsection \ref{Rem1}, so we focus here on the range $\, 0 <  \alpha < 1\,$. 

 \smallskip
For  this range of values of the skewness parameter, we choose to write the equation (\ref{SBBBM}) in terms of the right-continuous local time of the unknown process at the origin, namely
 \begin{equation}
\label{HSBB}
 Y(\cdot) \,=\, y_0  - \lambda \int_0^{\, \cdot} \text{sgn} \big( Y(t)\big) \, \mathrm{d} t+ W(\cdot) + { \, 2\, \alpha -1\, \over \alpha }\, L^Y (\cdot)\, .
\end{equation}
This equation is   of the more general form 
\begin{equation}
\label{LG}
Y(\cdot) \,=\, y_0 + \int_0^\cdot \taub (Y(t))\, \mathrm{d} W(t) +\int_\R \, L^Y (\cdot\,, \xi)\, \nub (\mathrm{d} \xi)\,, 
 \end{equation}
 with dispersion $\, \taub (y) \equiv 1\,$ and  measure   $$\, \,\nu (\mathrm{d} y)\, =\, 2\,\mathfrak{ b} (y)\, \mathrm{d} y+ \frac{\,2\, \alpha -1\,}{ \alpha} \,\deltab_0 (\mathrm{d} y)\,$$ with $\, \mathfrak{ b} (y)= - \lambda \, \text{sgn} (y)\,,~~ y \in \R\,$ as in (\ref{BBD}) above,  and with $\,\deltab_0 (\cdot)\,$   the Dirac mass at the origin.

\smallskip
We shall deal with the equation (\ref{HSBB}) using a direct methodology that removes the parts of finite variation, that is, both the drift and the local time, and ``reduces'' (\ref{HSBB}) to a stochastic differential equation in natural scale
\begin{equation}
\label{SDE_Y}
Z (\cdot) \,=\, \mathfrak{p} (y_0) + \int_0^{\, \cdot} \mathfrak{  s} \big( Z(t) \big)\, \mathrm{d} W(t) 
\end{equation}
for appropriate functions $\, \mathfrak{p} (\cdot)\,$ and $\, \mathfrak{s} (\cdot)\,$. This  approach was pioneered for the skew Brownian motion itself (i.e., with $\, \lambda =0\,$) by \textsc{Harrison \& Shepp} \cite{MR606993} 
, and for more general equations of the form (\ref{LG})  for suitable measurable functions $\, \taub (\cdot)\,$  and measures $\, \nub \,$ on $\, \mathcal{B} (\R)\,$, by  \textsc{Le$\,$Gall} \cite{MR0770393} \cite{MR777514} 
and \textsc{Engelbert \& Schmidt} \cite{MR798317}. 
The results in these works do not seem to cover the equation (\ref{HSBB}),   but those in \textsc{Bass  \& Chen} \cite{MR2203887} 
do; we have preferred to detail a direct construction which is, in our opinion at least, quite simpler. 
 
 \smallskip
 In this spirit, let us introduce the scale function 
$$
\mathfrak{ p} (y)= { \, 1 - \alpha\, \over 2\, \lambda} \left( e^{\, 2 \lambda y} - 1\right), ~~~y>0\,;\qquad \mathfrak{ p} (0)=0\,; \qquad \mathfrak{ p} (y)= { \,   \alpha\, \over 2\, \lambda} \left( 1- e^{\, - 2 \lambda y}  \right) , ~~~y<0\,.
$$
This has left-continuous derivative
$$ 
\mathfrak{p}^\prime (y) \,=\, ( 1 - \alpha )\,e^{\, 2 \lambda y}\,  \mathbf{ 1}_{ (0, \infty)} (y) + \alpha \,e^{\, - 2 \lambda y}\,  \mathbf{ 1}_{ (-\infty, 0]} (y)\,, \qquad y \in \R 
$$
which is bounded away from zero, and   second derivative measure
$$
\mathfrak{p}^{\prime \prime} (\mathrm{d}y) \,=\, - 2 \,\mathfrak{ b} (y) \, \mathfrak{p}' (y)\, \mathrm{d}y + \big( 1- 2 \, \alpha \big) \, \deltab_0 (\mathrm{d}y)\, .
$$
Likewise, we introduce the inverse 
$$
\mathfrak{q} (z)= { \, 1 \, \over \, 2\, \lambda\, } \log  \left( 1 + {\, 2\, \lambda \, z\,  \over \, 1 - \alpha\, } \right), ~~~z>0\,;\quad \mathfrak{q} (0)=0\,; \quad \mathfrak{q} (z)= { \, -1 \, \over \, 2\, \lambda\, } \log  \left( 1 - {\, 2\, \lambda \, z\,  \over \,   \alpha\, } \right), ~~~z<0
$$
of the function $\,\mathfrak{p}(\cdot)\,$, as well as its left-continuous derivative and its second-derivative measure 
$$ 
\mathfrak{q}^{\prime} (z) \,=\,  \big( 1 - \alpha  +  2\, \lambda \, z \big)^{-1}  \, \mathbf{ 1}_{ (0, \infty)} (z) + \big(   \alpha  -  2\, \lambda \, z \big)^{-1} \,\mathbf{ 1}_{ (-\infty, 0]} (z)\,,
$$
$$
  \mathfrak{q}^{\prime \prime} (\mathrm{d} z) \,=\, 2\,  \mathfrak{b}  (z) \, \big( \mathfrak{q}^{\prime} (z) \big)^2 \, \mathrm{d} z + 
    { \, 2\, \alpha -1\, \over \, \alpha \, ( 1 - \alpha)\,} \,\deltab_0 (\mathrm{d} z)\,.
$$
 
 \medskip
 \noindent
{\it Analysis:}  Assume that a solution to (\ref{HSBB}) has been constructed; in particular, the process $\, Y(\cdot)\,$ is then a continuous semimartingale for which (\ref{ZS})  holds a.s. We look at the process $\, Z(\cdot) : =  \mathfrak{p} (Y(\cdot))\,$ and apply the \textsc{It\^o-Tanaka} rule
$$
\mathrm{d} Z(t) \,=\, \mathfrak{p}^{\prime} (Y(t))\, \mathrm{d} Y(t) + \mathrm{d} \int_\R \,  L^Y (t, y) \,  \mathfrak{p}^{\prime \prime} (  \mathrm{d} y)  
$$
of (\ref{ITOT1}),  to obtain
  $$
\mathrm{d} Z(t) \,=\, \mathfrak{p}^{\prime} (Y(t))\,\Big[\, \mathfrak{b} (Y(t))\, \mathrm{d} t+ \mathrm{d} W(t) + { \, 2\, \alpha -1\, \over \alpha }\,\mathrm{d} L^Y (t)\, \Big]~~~~~~~~~~~~~~~~~~~~~~~~~~~~~~~~~~~~~~~
$$  
$$
~~~~~~~~~~~~~~~~~~~~ -   \mathrm{d}  \int_\R 2\, \mathfrak{b}  (y)\, \mathfrak{p}^{\prime} (y) \, L^Y (t, y)\, \mathrm{d} y \,+\, (1 - 2\, \alpha )\, \mathrm{d} L^Y (t)\,=\, \mathfrak{p}^{\prime} \big( \mathfrak{q} (Y(t) ) \big)\, \mathrm{d} W(t)\,.
$$
We have  used here the occupation-time-density formula (\ref{OTD}), and the property  $\, \mathfrak{p}^{\prime}  (0) = \alpha\,$. Now the piecewise-linear function 
\begin{equation}
\label{NewS}
\mathfrak{s}(z)\,:=\, \mathfrak{p}^{\prime} \big( \mathfrak{q}  (z) \big)\,=\, { 1 \over \,\mathfrak{q}^{\prime}(z)\,} \,=\,\big( 1 - \alpha  +  2\, \lambda \, z \big)   \, \mathbf{ 1}_{ (0, \infty)} (z) + \big(   \alpha  -  2\, \lambda \, z \big)  \,\mathbf{ 1}_{ (-\infty, 0]} (z)\,, \quad z \in \R~~
\end{equation}
is  bounded away from the origin, so the process $\, Z(\cdot)\,$ is the pathwise unique, strong solution of the stochastic differential equation (\ref{SDE_Y}) for this new dispersion function (\textsc{Nakao} \cite{MR0326840}
); and because  $\, Y(\cdot)\,$ and $\, Z(\cdot)\,$ are bijections of each other, we have again the filtration identities
\begin{equation}
\label{strong}
 \mathfrak{F}^Z (t) \,=\,  \mathfrak{F}^Y (t) \,=\, \mathfrak{F}^W (t)\,, \qquad \, \,0 \le t < \infty\, .
 \end{equation}

\medskip
 \noindent
{\it Synthesis:} Consider the strong solution $\,Z(\cdot)\,$ of the stochastic differential equation (\ref{SDE_Y}) with the new dispersion function of (\ref{NewS}),  and define the process $\, Y(\cdot) : =  \mathfrak{q} (Z(\cdot))\,$. Since $\, \mathrm{d} Z(t) = \mathfrak{s} (Z(t))\, \mathrm{d} W(t)\,$ and 
$\, \mathrm{d} \langle Z \rangle (t) = \mathfrak{s}^2 (Z(t))\, \mathrm{d} t\,$, this process satisfies almost surely 
$$
\int_0^\infty \mathbf{ 1}_{ \{ Z(t)=0 \} }\, \mathrm{d} t\,=\,\int_0^\infty \mathbf{ 1}_{ \{ Z(t)=0 \} }\, { \, \mathrm{d} \langle Z \rangle (t)\, \over \, \mathfrak{ s}^2 (Z(t))\,} \,\le \, \big( \min (\alpha, \, 1- \alpha  ) \big)^{-2}\, \int_0^\infty \mathbf{ 1}_{ \{ Z(t)=0 \} }\, \mathrm{d} \langle Z \rangle (t)\,=\,0\,,  
$$
and we apply the \textsc{It\^o-Tanaka} rule to obtain
$$
\mathrm{d} Y(t) \,=\, \mathfrak{q}^{\prime} (Z(t))\, \mathrm{d} Z(t) + \mathrm{d}  \int_\R     L^Z (t, z) \,  \mathfrak{q}^{\prime \prime} (  \mathrm{d} z) \,.
$$
On the strength of the occupation-time-density formula and $\, \mathfrak{s}(\cdot) \,\mathfrak{q}^{\prime}(\cdot) \equiv 1\,$, this gives
$$
\mathrm{d} Y(t)\,=\,  \mathrm{d} W(t) +  \mathrm{d}  \int_\R 2\, \mathfrak{b}  (y)\, \big( \mathfrak{q}^{\prime} (y) \big)^2 \, L^Y (t, y)\, \mathrm{d} y \,+\, {\,   2\, \alpha -1\, \over \, \alpha \, (1 - \alpha)\,} \, \mathrm{d} L^Z (t) 
$$
$$
 =\,  \mathrm{d} W(t) \,+   \, 
 \mathfrak{b}  (Y(t))\, \big( \mathfrak{q}^{\prime} (Y(t)) \big)^2 \, \mathfrak{s}^{2} (Y(t))  \, \mathrm{d} t 
  \,+\,
  {\,   2\, \alpha -1\, \over \, \alpha \, (1 - \alpha)\,}  \, \mathrm{d} L^Z (t)
 $$
$$
 =\,  \mathrm{d} W(t) \,+   \, \mathfrak{b}  (Y(t))\, \mathrm{d} t \,+\, 
  {\,   2\, \alpha -1\, \over \, \alpha \,  }  \, \mathrm{d} L^Y (t) \,,
$$

\medskip
\noindent
that is, the equation (\ref{HSBB}) of the skew Brownian motion with bang-bang drift for the process $\, Y(\cdot)\,$. We have used here  the comparison of the local times at the origin for these two processes: 
\begin{equation}
\label{LXLY}
  L^Z (\cdot\, )\,=\, ( 1 - \alpha ) \, L^Y (\cdot) \,.
  \end{equation} 
  
 This last identity (\ref{LXLY}) can be justified as follows: We start by noting  
$$
L^Z (\cdot  )=\, \lim_{\varepsilon \downarrow 0} \, { 1 \over \, 2 \, \varepsilon \,} \int_0^\cdot \mathbf{ 1}_{ \{ 0< Z(t)< \varepsilon \} }\, \mathrm{d} \langle Z \rangle (t) =\,  \lim_{\varepsilon \downarrow 0} \, { 1 \over \, 2 \, \varepsilon \,} \int_0^\cdot \mathbf{ 1}_{ \{ 0< Y(t)< \mathfrak{ q}(\varepsilon) \} }\, \big( \mathfrak{ p}^\prime (Y(t)) \big)^2 \, \mathrm{d}   t   
$$
$$
=\,(1 - \alpha)\cdot \lim_{\varepsilon \downarrow 0} \, { \,1- \alpha\, \over \, 2 \, \varepsilon \,} \int_0^\cdot \mathbf{ 1}_{ \{ 0< Y(t)< \mathfrak{ q}(\varepsilon) \} }\, \left( \frac{\,\mathfrak{ p}^\prime (Y(t))\,}{1-\alpha} \right)^2 \mathrm{d}   t   \,.
$$
On the event $\, \{ 0< Y(t)< \mathfrak{ q}(\varepsilon) \}\,$ we have 
$$
1 \, \le \, \frac{\,\mathfrak{ p}^\prime (Y(t))\,}{1-\alpha} \,\le\, e^{\, 2 \lambda \,\mathfrak{ q}(\varepsilon) }\,; \quad \hbox{and since}~~~\lim_{\varepsilon \downarrow 0} \left( \frac{\,\mathfrak{ q}(\varepsilon)\,}{\varepsilon} \right)\,=\, \frac{1}{\, 1 - \alpha\,}\,,
$$ 
we deduce the claimed identity of (\ref{LXLY}), namely
$$
L^Z (\cdot  )=\,(1 - \alpha)\cdot \lim_{\varepsilon \downarrow 0} \, { \,1 \, \over \, 2\, \mathfrak{ q}(\varepsilon)\,   } \int_0^\cdot \mathbf{ 1}_{ \{ 0< Y(t)< \mathfrak{ q}(\varepsilon) \} }\,   \mathrm{d}   t  \,= \,( 1 - \alpha ) \, L^Y (\cdot)\,.
$$

\medskip
\noindent
$\bullet~$ Taken together, the Analysis and Synthesis parts of this argument establish the following result. 

\begin{thm}
  \label{Theorem 3}
The equation (\ref{SBBBM}) admits a pathwise unique, strong solution for all values of its ``skewness parameter'' $  \alpha \in [0, 1]\,$, and  we have the filtration identity $ \,  \mathfrak{F}^Y (t)= \mathfrak{F}^W (t)\,, ~   \,0 \le t < \infty $ in (\ref{YW}). 
 \end{thm}

\begin{rem}
 \label{HS}
  As shown towards the end of section 3 in \textsc{Harrison \& Shepp} \cite{MR606993} 
  , the stochastic equation (\ref{SBBBM}) admits no solution for $\, \alpha \notin [0,1]\,$. Consequently, when $\, \eta + \zeta \neq 0\,$ holds but  the   condition (\ref{1.5}) fails because $\, \alpha = \eta / (\eta + \zeta)  \notin [0,1]\,$, the system of equations (\ref{1.3}), (\ref{1.4}) admits no solution. 
\end{rem}

\begin{rem}
 \label{SMF}
 We compute in the next subsection the transition probabilities of the diffusion   $\, Y(\cdot)\,$. It follows from these computations, and in conjunction with the theory developed in \textsc{Portenko} 
 \cite{MR522237} 
\cite{MR532446} 
\cite{MR1104660}
, that  this process has the strong \textsc{Markov} and \textsc{Feller} properties. 
\end{rem}

From the Remarks \ref{EZ0}, \ref{EZ1}, \ref{HS} and in conjunction with Theorem \ref{Theorem 1}, we obtain now the following result.

 \begin{prop}
  \label{Proposition 3}
The conditions of (\ref{1.5}) are not just sufficient but also necessary for the well-posedness of the system of equations (\ref{1.3}), (\ref{1.4}). 
 \end{prop}

\smallskip

\subsection{Joint Distribution of SBBBM and its Local Time}
\label{sec5.1}

\smallskip
\noindent
Let us recall a construction of the skew Brownian motion (\textsc{It\^o  \& Mc$\,$Kean} \cite{MR0154338}
, \cite{MR0345224}
, \textsc{Walsh} \cite{TempsLocaux78}
). We take a  Brownian motion starting from $\,y_{0} \ge 0\,$, reflect it at the origin, and consider its excursions away from the origin. Then we change the sign of each excursion independently with probability $\,1-\alpha \in (0, 1) \,$. The resulting process is positive with probability $\,\alpha\,$, negative with probability $\,1 - \alpha\,$. This implies a {\it non-symmetric  reflection principle} around the origin. We shall see that even in the presence of the bang-bang drifts $$\, \mathfrak{ b} (y)\,=\, - \lambda \,\text{sgn}  (y)\,,~~ ~~y \in \R\,$$
as in (\ref{BBD}), this principle continues to hold  for the skew Brownian motion. Thus, the joint distribution of SBBBM and its local time are derived here. 

\smallskip
By \textsc{Girsanov}'s theorem (e.g., \textsc{Karatzas \& Shreve} \cite{MR1121940}
, section 3.5), we consider the ``reference probability measure'' $\,\mathbb P_{\star}\,$, under which the process $     \,\int^{\cdot}_{0}\mathfrak b (Y(t))\, {\mathrm d} t + W(\cdot)\,$ becomes standard Brownian motion.  For every given $\,t \in [0, \infty)\,$ the \textsc{Radon-Nikod\'ym} derivative on $\, {\mathfrak F^{Y}(t)}\,$ of the original measure with respect to the reference measure, is 
\[
\frac{ {\mathrm d} \mathbb P\, \, }{\, {\mathrm d} \mathbb P_{\star} \, } \bigg \vert_{\mathfrak F^{Y}(t)}\, = \, \exp \Big\{   \int^{t}_{0} \mathfrak b (Y(s)) {\mathrm d} W(s) + \frac{\, 1 \, }{2} \int^{t}_{0} \mathfrak b^2 (Y(s)) {\mathrm d} s \Big \} 
\]
\[
~~~~~~~~\,~~
= \,\exp \Big\{ \lambda \big( \lvert y_0 \rvert - \lvert Y(t) \rvert + 2 \widehat{L}^{\,Y}(t) \big) - \frac{\, \lambda^{2}\, }{2} t \Big \} \,  ; 
\]

\medskip
\noindent
we have used in this last equation  the relationships (\ref{2.13}), (\ref{2.10}). 

Under the reference probability measure $\,\mathbb P_{\star}\,$, the process $\,Y(\cdot)\,$ is   skew Brownian motion  which starting at  $\,y_{0}\,$. As shown in \textsc{Walsh} \cite{TempsLocaux78} 
(see also \textsc{Lang} \cite{MR1340556}
), the transition probability density function $\,\mathfrak p_{\star } (t; y_{0}, \xi ) = \mathbb P_{\star} (Y(t) \in {\mathrm d} \xi ) \, /\, {\mathrm d} \xi  \,$ for this process is given by 
\[
\mathfrak p_{\star } (t; y_{0}, \xi ) \, =\, \frac{1}{\, \sqrt{2 \pi t}\, } \exp \Big \{ - \frac{\,(y_{0}-\xi )^{2}\, }{2 t} \Big\}  + (2 \alpha - 1) \cdot \text{sgn} (\xi ) \cdot \frac{1}{\sqrt{2 \pi t}} \exp \Big\{ - \frac{ \, (\lvert y_{0} \rvert + \lvert \xi  \rvert )^{2}\, }{2 t}\Big\} \,  
\]
for $\, (\xi , y) \in \R^{2}\,$, $\, t  > 0\,$. Moreover, by the method of elastic Brownian motion (e.g.,  \textsc{Karatzas \& Shreve} \cite{MR744236},  
\textsc{Appuhamillage et al.} 
\cite{MR2759199}) 
the joint distribution of the skew Brownian motion and its symmetric local time is computed as

\[
\hspace{-11cm} 
\mathbb P_{\star} ( Y(t) \in {\mathrm d} \xi  \, , \, 2 \widehat{L}^{Y}(t) \in {\mathrm d} b)  \,=
\]
\[
= \,\big \{ 1 + (2 \alpha - 1) \, \text{sgn} (\xi )  \big \} \cdot \frac{\, \lvert \xi  \rvert + b + \lvert y_{0}\rvert\, }{ \sqrt{2 \pi t^{3}} } \exp \Big \{ - \frac{\, ( \lvert \xi \rvert + b + \lvert y_{0}\rvert )^{2} \,  }{2 \, t} \Big\} \, {\mathrm d} \xi  \, {\mathrm d} b \, ; \quad   b  > 0  
\]

\medskip
\noindent
and
\[
\mathbb P_{\star} \big( Y(t) \in {\mathrm d} \xi  \, ,\, 2 \widehat{L}^{Y}(t) \, =\, 0\big) \, =\, 
\frac{1}{\, \sqrt{2 \pi t}\, } \Big[ \exp \Big \{ - \frac{\,(\lvert y_{0} \rvert - \lvert \xi  \rvert)^{2}\, }{2 t} \Big\}  - \exp \Big \{ - \frac{\,(\lvert y_{0} \rvert + \lvert \xi  \rvert )^{2}\, }{2 t} \Big\} \Big] {\mathrm d} \xi     
\]
for $\,\xi  \in \R\,$. Note that we have $$\, 1 + (2 \alpha - 1) \, \text{sgn} (\xi ) \, =\, 2 \alpha ~~~~  \mathrm{if}~~  \,\xi  > 0\,,~~~~ \mathrm{and} ~~~~\,1 + (2 \alpha - 1) \, \text{sgn} (\xi ) \, =\,2 (1-\alpha) \,~~ \mathrm{if}~~\,\xi  \le 0\,.$$ Thus, the non-symmetric reflection principle around the origin works intuitively even for the joint distribution. As expected, when there is   no accumulation of local time   at the origin,  the skewness parameter $\,\alpha\,$ does not affect the transition probabilities.  

\medskip
\noindent
$\bullet~$ 
We bring the above formulae from the reference measure $\,\mathbb P_{\star}\,$ back to the original measure $\,\mathbb P\,$, by means of 

\begin{equation} 
\label{eq: joint}
\hspace{-6cm} \mathbb P \big (Y^{\pm}(t) \in A\, , Y^{\mp}(t) \, =\, 0\, , \, 2 \widehat{L}^{Y}(t) \in B \big) \,=~~~~~~~~~~~~~~~~~~~~~~~~~
\end{equation}
\[
~~~~~~~~~~~~~\, =\,  
\exp \big\{ \lambda   \lvert y_0 \rvert   -   \lambda^{2}\,t /  2  \big \}
\cdot \mathbb E^{\,\P_{\star}} \Big[\, \exp \big\{ 2 \,\widehat{L}^{Y}(t) - Y^{\pm}(t) \big \} \cdot \,  {\bf 1}_{\{ Y^{\pm}(t) \in A\, , Y^{\mp}(t) \, =\, 0\,,  \,2 \widehat{L}^{Y}(t) \in B\}} \,\Big]  \, 
\]

\medskip
\noindent
for $\, (A, B) \in \mathcal B (\R) \times \mathcal B(\R) \,$, $\, \,t > 0 \,$. With $\, (\xi , b) \in [0, \infty) \times (0, \infty) \,$, the joint density functions are 
\begin{equation} \label{eq: joint+}
\hspace{-6cm} \mathbb P ( Y^{+}(t) \in {\mathrm d} \xi  \, , Y^{-}(t) \, =\, 0\, , \, 2 \widehat{L}^{Y}(t) \in {\mathrm d} b ) \, = \, 
\end{equation}
\[
\, =\, 2 \alpha \cdot e^{\,-2\lambda \xi} \cdot 
\frac{\, \xi  + b + \lvert y_{0}\rvert\, }{ \sqrt{2 \pi t^{3}} } \exp \Big \{ - \frac{\, ( \xi  + b + \lvert y_{0}\rvert - \lambda t)^{2} \,  }{2 \, t} \Big\} \, {\mathrm d} \xi  \, {\mathrm d} b\, , 
\]
as well as 

\begin{equation} 
 \label{eq: joint-}
\hspace{-6cm} \mathbb P ( Y^{-}(t) \in {\mathrm d} \xi  \, , Y^{+}(t) \, =\, 0\, , \, 2 \widehat{L}^{Y}(t) \in {\mathrm d} b ) \, =\, 
\end{equation}
\[
\, =\, 2(1 -  \alpha) \cdot e^{\,-2\lambda \xi} \cdot 
\frac{\, \xi  + b + \lvert y_{0}\rvert\, }{ \sqrt{2 \pi t^{3}} } \exp \Big \{ - \frac{\, ( \xi  + b + \lvert y_{0}\rvert - \lambda t)^{2} \,  }{2 \, t} \Big\} \, {\mathrm d} \xi  \, {\mathrm d} b\, .
\]

\medskip
\noindent
Whereas, when there is no  accumulation of local time, we have 
\begin{equation} \label{eq: joint0}
\hspace{-6cm} \mathbb P \, \big( \,Y^{\pm}(t) \in {\mathrm d} \xi  \, , Y^{\mp}(t) \, =\, 0\, , \, 2 \widehat{L}^{Y}(t) \, =\,  0\,\big) \,=
\end{equation}
\[
=\, \frac{1}{\, \sqrt{2\pi t}\, } \Big( \exp \Big \{ - \frac{\,(\xi - \lvert y_{0}\rvert + \lambda t )^{2}\, }{2 t} \Big\} - e^{- 2\lambda \, \xi } \cdot \exp \Big \{ - \frac{\,(\xi  +  \lvert y_{0} \rvert + \lambda t )^{2}\, }{2 t} \Big\} \Big) {\mathrm d} \xi  \, , \quad \xi  > 0 \, . 
\]

\bigskip
\noindent
$\bullet~$ 
The marginal density $\,\mathfrak p  (t; y_{0}, \xi ) \, {\mathrm d} \xi = \mathbb P  (Y(t) \in {\mathrm d} \xi )   \,$  of $\, Y(t)\,$ under the original probability measure $\,\mathbb P \,$ is obtained from 
\begin{equation} 
\label{eq: TDF}
\mathbb P ( Y(t) \in {\mathrm d} \xi  ) = \mathbb P ( Y(t) \in {\mathrm d} \xi  \, , \, 2 \widehat{L}^{Y}(t) > 0) + \mathbb P (Y(t) \in {\mathrm d} \xi  \, , 2 \widehat{L}^{Y}(t) \, =\, 0) \cdot {\bf 1}_{\{\xi  \,y_{0} > 0\}} \, , 
\end{equation}
where the second term takes care of the case when the local time is absent. If $\, \xi  > 0\,$ and $\,y_{0} > 0\,$, the marginal density becomes 
\begin{equation} 
\label{TDF1}
\mathfrak p  (t; y_{0}, \xi )\, =\,
(2 \alpha - 1)\,  e^{-2 \lambda \xi } \cdot \frac{1}{\, \sqrt{2 \pi t}\, } \exp \Big \{ - \frac{\,(  \xi  +  y_{0}  - \lambda t )^{2}\, }{2 t} \Big\}\,  ~~~~~~~~~~~~~~~~~~~~~~~~~~~~~~
\end{equation}
\[
~~~~~~~~~~~~~~~~~~~~~
+\, \frac{1}{\, \sqrt{2 \pi t}\, } \exp \Big \{ - \frac{\,(  \xi  -  y_{0}+ \lambda t )^{2}\, }{2 t} \Big\}\,+ \,(2 \alpha)  \cdot \frac{\lambda \, e^{-2 \lambda \xi  }}{\, \sqrt{2 \pi t}\, } \int^{\infty}_{ \xi  + y_{0}} e^{\,- \frac{(u-\lambda t)^{2}}{2 t }} {\mathrm d} u \, ;   
\]

\medskip
\noindent
whereas, if $\,\,\xi  <0\,$ and $\, y_{0} < 0\,$,   this expression  becomes  
\begin{equation} 
\label{TDF2}
\mathfrak p  (t; y_{0}, \xi )\, =\,
 (1 - 2 \alpha )  e^{2 \lambda \xi } \cdot \frac{1}{\, \sqrt{2 \pi t}\, }  \exp \Big \{ - \frac{\,( - \xi  -  y_{0}  - \lambda t )^{2}\, }{2 t} \Big\} ~~~~~~~~~~~~~~~~~~~~~~~~~~~~
\end{equation}
\[
~~~~~~~~~~~~~~~~~~~~~~~~~~~~~~~~~+ \frac{1}{\, \sqrt{2 \pi t}\, }  \exp \Big \{ - \frac{\,( - \xi  +  y_{0}  + \lambda t )^{2}\, }{2 t} \Big\} + \,2\, ( 1- \alpha)  \cdot \frac{\lambda \, e^{2 \lambda \xi  }}{\, \sqrt{2 \pi t}\, } \int^{\infty}_{ -\xi  - y_{0}} e^{- \frac{(u-\lambda t)^{2}}{2 t }} {\mathrm d} u \, .   
\]

\noindent
If $\, \xi  \, y_{0} \le 0\,$, then this expression  becomes  
$$
\mathfrak p  (t; y_{0}, \xi )\,=\,
 \big \{ 1 + (2 \alpha - 1) \, \text{sgn} (\xi )  \big \} \cdot \frac{\, e^{- 2 \lambda \lvert \xi  \rvert} \, }{ \, \sqrt{2 \pi t} \, } \cdot  ~~~~~~~~~~~~~~~~~~~~~~~~~~~~~~~~~~~~~~~~~~~~~~~~~~~~~~~~~
$$
\begin{equation} 
\label{TDF3}
~~~~~~~~~~~~~~~~~~~~~~~~~~~~~~~~~~~~~~ \cdot \Big[  \exp \Big \{ - \frac{\, ( \lvert \xi \rvert + \lvert y_{0}\rvert  - \lambda t )^{2} \,  }{2 \, t} \Big\} 
 + \lambda \, \int^{\infty}_{ \lvert \xi \rvert + \lvert y_{0}\rvert} e^{- \frac{(u-\lambda t)^{2}}{2 t }} {\mathrm d} u \Big] \, , ~~~~~
 \end{equation}
 where the term $\,\mathbb P (Y(t) \in {\mathrm d} \xi  \, , 2 \widehat{L}^{Y}(t) \, =\, 0) \,$  in (\ref{eq: TDF}) is now equal to zero.

\medskip 

\begin{rem} \label{rem: STR} Letting $\,t \to \infty\,$ in (\ref{TDF1})-(\ref{TDF3}), we derive for the process $\,Y(\cdot)\,$ the stationary measure $\,\mathfrak m (\cdot) = \int_{\cdot} \mathfrak p_{\infty}(\xi )\, \mathrm{d} \xi  \,$     with the double-exponential probability density function
 \begin{equation} 
 \label{eq: sta dens}
\mathfrak p_{\infty}(\xi ) \, := \,\lim_{t \to \infty} \mathfrak p (t; y_{0}, \xi ) \, =\, \, \alpha \cdot (2\lambda)\,  e^{\,- 2 \, \lambda \,  \xi}   \cdot {\bf 1}_{\{\xi  > 0\}} + (1-\alpha) \cdot (2\lambda) \, e^{ \,2 \,\lambda \, \xi  } \cdot {\bf 1}_{\{\xi  \le 0\}} \, .  
\end{equation}
It can be verified that (\ref{eq: sta dens}) is the invariant distribution. Furthermore, it follows from the transition density (\ref{TDF1})-(\ref{TDF3}) and the stationary distribution (\ref{eq: sta dens}) that the following duality holds: 
\[
\int_{\R} g(y) \Big( \int_{\R} f(\xi ) \mathfrak p (t; y,\xi ) {\mathrm d} \xi  \Big) \mathfrak p_{\infty} (y) {\mathrm d} y \, =\,  \int_{\R} f(\xi ) \Big( \int_{\R} g(y) \mathfrak p (t; \xi ,y) {\mathrm d} y \Big) \mathfrak p_{\infty}(\xi ) {\mathrm d} \xi  , 
\]
for arbitrary bounded, measurable functions $\,f, g\,$, or equivalently 
\begin{equation} \label{eq: sym rela}
\int_{\R} g(y) \, \mathbb E_{y} \big[ f(Y(t))] \, \mathfrak p_{\infty} (y) \, {\mathrm d} y \, =\, 
\int_{\R} f(\xi ) \, \mathbb E_{\xi } \big[ g(Y(t))] \, \mathfrak p_{\infty} (\xi ) \,{\mathrm d} \xi  \, ; \quad t > 0 \,   .
\end{equation}
Here $\,\mathbb E_{y} \,$ stands for the expectation under the measure $\,\mathbb P_{y}\,$ induced by $\,Y(\cdot) \,$ which starts from $\,y \in \R \,$. Thus,  under the probability measure $$\,\mathbb P_{\infty}(\cdot) \, :=\, \int_{\R} \mathbb P_{y} (\cdot)\, \mathfrak m( {\mathrm d} y) \,,$$ the process $\,Y(\cdot)\,$ is stationary; and moreover, given a fixed time $\,T \in ( 0, \infty) \,$, the time reversal 
\begin{equation}
\label{TiRev}
  \widehat{Y}(t) \, :=\, Y(T-t)\, , \qquad \,0 \le t \le T 
\end{equation}
satisfies 
\begin{equation} 
\label{eq: time symmetry}
\mathbb E^{\,\mathbb P_{\infty}} \big[ \, f_{0}(Y(t_{0})) \cdots f_{n} (Y(t_{n})) \, \big] \, =\, 
\mathbb E^{\,\mathbb P_{\infty}} \big[\, f_{0}( \widehat{Y}(t_{n})) \cdots f_{n} ( \widehat{Y}(t_{0})) \,\big]   
\end{equation}

\medskip
\noindent
  for every integer  $\,n \in \N\,$, collection of time points $\, 0 \, =\,  t_{0} < t_{1} < \cdots < t_{n} \, =\, T$,  and    bounded, measurable functions $\,f_{0}, \ldots , f_{n}\,$.   
\end{rem}

\begin{rem} \label{rem: FS} 
The infinitesimal generator of the process $\,Y(\cdot)\,$ may be defined formally by 
\begin{equation}
[\mathcal L f] (\xi)\, :=\, \frac{1}{\,2\,} \,f^{\prime\prime}(\xi)  - \lambda \,   
{\text{sgn}}  (\xi) f^{\prime}(\xi) +    2 \,(2 \alpha - 1) f^{\prime}(\xi) \, \delta_0 (\xi) \, , \quad \xi \in \R\,  
\end{equation} 
for $\,f \in \mathcal D\, :=\, C^{\infty}_{0}(\R)\,$, where $\,\delta_0 (\cdot)\,$ is the ``Dirac delta function'' at the origin. Here we use the parametrization for the symmetric local time $\,\widehat{L}^{Y}(\cdot)\,$. 
Let us denote formally the symmetric version of the density of $\,\mathfrak m\,$ by $$\, \widehat{\mathfrak{p}}_{\infty}(\xi ) \,  =\, (2\alpha) \, \lambda \, e^{-\lambda \xi} \cdot {\bf 1}_{\{ \xi > 0\}} + 2(1 - \alpha) \, \lambda \, e^{\lambda \xi} \cdot {\bf 1}_{\{ \xi < 0\}} + \lambda \cdot {\bf 1}_{ \{\xi \, =\,  0\}} \,.$$ 
Then by direct calculation 
\begin{equation} \label{eq: FS (1.6)}
\int_{\R} f (\xi) \, [\mathcal L g] (\xi) \, \mathfrak m ({\mathrm d}\xi) \, =\, \int_{\R} g(\xi)  \, [\mathcal L f](\xi) \, \mathfrak m ( {\mathrm d} \xi) \, , \quad \int_{\R} [\mathcal L f ]  \, \mathfrak m ({\mathrm d} \xi) \, =\, 0 \, ; \quad f, \,g \in \mathcal D \,  . 
\end{equation}
Applying Theorem 2.3 of \textsc{Fukushima \& Stroock} \cite{MR875449} 
, we arrive at the same conclusion (\ref{eq: time symmetry}). 
\end{rem}

\begin{rem} \label{rem: TR} Let us define the time reversal of $\,Y(\cdot) \,$ as in (\ref{TiRev}).  
Following \textsc{Pardoux} \cite{MR942014}   
and \textsc{Petit} \cite{MR1464173} 
, we may show that the time reversal is a solution of the stochastic equation 
$$
\widehat{Y}(t) \, =\, \widehat{Y}(0) + {W}^{\sharp}(t) + 2 \, (1-2\alpha) \,\widehat{L}^{\, \widehat{Y}}(t) +~~~~~~~~~~~~~~~~~~~~~~~~~~~~~~
$$
\begin{equation} 
\label{eq: Yh}
~~~~~~~~~~~~~~~~~~~~~~~~~~~~~~~+  \int^{t}_{0} \left( \lambda \,   
{\text{sgn}} \big( \,\widehat{Y}(s) \big) +   \frac{\partial }{\, \partial \xi } \log \mathfrak p  \big(T-s; y_{0}, \widehat{Y}(s) \big) \right)  {\mathrm d} s  
\end{equation}
for $\,0 \le t \le T\,$, where $\, {W}^{\sharp}(\cdot)\,$ is a standard Brownian motion with respect to the backwards filtration $\, \mathbf{F}^{ \,\widehat{Y}}(\cdot) \,$ generated by the time-reversed process $\, \widehat{Y}(\cdot)\,$  of (\ref{TiRev}), and 
\begin{equation} \label{eq: Lhh}
\, \widehat{L}^{ \, \widehat{Y}}( t) \, :=\,  \widehat{L}^{\,Y}(T) - \widehat{L}^{\,Y}(T- t \,)\,, \qquad 0 \le t \le T\,. 
\end{equation}
In the special case $\,y_{0} = 0 = \widehat{Y}(T)\,$, the logarithmic derivative  of the transition probability density function is 
\[
\frac{\partial}{\, \partial \xi \, } \log \mathfrak p  \big(t; 0, \xi \big) \,= \,-\, 2\,  \lambda \, \text{sgn} (\xi ) - \frac{\,\xi \,}{\, t \, }\cdot  \frac{\mathfrak C_{1}(t,\xi )}{\, \mathfrak C_{1}(t,\xi ) + \mathfrak C_{2}(t,\xi ) \, } \, , 
\]
where 
\[
\mathfrak C_{1} (t,\xi ) \, :=\,  \exp \Big( - \frac{( \lvert \xi  \rvert + \lambda t)^{2}}{2t} \Big) \, , \quad 
\mathfrak C_{2} (t,\xi )\, :=\,  \lambda \, e^{- 2 \lambda \lvert \xi  \rvert} \int^{\infty}_{ \lvert \xi \rvert} \exp \Big( - \frac{ (u-\lambda t)^{2}}{2 t} \,\Big) \,{\mathrm d} u \, . 
\]
Thus, the time reversal is a {\it skew Brownian bridge} with bang-bang drift    
\[
\widehat{Y}(\cdot) \, =\,  \widehat{Y}(0) +   
W^{\sharp}(\cdot) + 2\, \big( 1- 2\,\alpha \big) \,\widehat{L}^{\, \widehat{Y}}(\cdot) -~~~~~~~~~~~~~~~~~~~~~~~~~~~~~~~~~~~~~~~~~~~~~~~~~~~~~~~~~~~~~~~
\]
\[
~~~~~~~~~~~~~~~~~~~~~~~~~~~~~~~~~~~~~~- \int^{\, \cdot}_{0} \bigg[\, \lambda \, 
{\text{sgn}} \big( \widehat{Y}(t) \big) +  \frac{ \widehat{Y}(t)}{T-t} \cdot \left( \frac{\mathfrak C_{1}}{\mathfrak C_{1}+ \mathfrak C_{2}} \right)   \big(T-t, \widehat{Y}(t) \big)   \bigg]  {\mathrm d} t \, .  
\]

\medskip
\noindent
This result suggests that the time reversal of SBBBM in general looks like a skew Brownian bridge drifted towards the target point $\, \widehat{Y}(T) \, =\, y_{0}\,$.
\end{rem}

\section{Applications of the Skew Representations (\ref{2.23b})-(\ref{2.24})}
\label{sec7}

\noindent
With the skew representations in section \ref{sec3.2} and the joint distribution of $\,(Y(\cdot), \widehat{L}^{\,Y}(\cdot))\,$ in section \ref{sec5.1} it is now straightforward to compute the transition density of the system (\ref{1.3})-(\ref{1.4}) as well as its time reversal.

\subsection{Transition Density}
\label{sec7.1} 

Let us discuss only some special cases, since the other cases are quite similar. 
For example, in the degenerate case with $\,\sigma \, =\, 0\,$, thus $\,\rho \, =\, 1\,$, $\,\gamma = 1\,$ and with $\, x_{1} \ge x_{2}\,$, (\ref{2.23b})-(\ref{2.24}) become 
\[
X_{1}(t) = x_{1} + g\, t + (Y^{+}(t)- y_{0}^{+})- \beta \, \widehat{L}^{Y}(t) \, , \quad 
X_{2}(t) = x_{2} + g\, t + (Y^{-}(t)-y_{0}^{-}) - \beta \, \widehat{L}^{Y}(t) \, , 
\]
for $\, 0 \le t < \infty \,$, where $\,y_{0} \, =\, x_{1} - x_{2} \ge 0\,$, and hence the transition density of $\,(X_{1}(\cdot), X_{2}(\cdot)) \,$ is 
\[
\mathbb P \big (X_{1}(t) \in {\mathrm d} \xi_{1}\, , \, \, X_{2}(t) \in {\mathrm d} \xi_{2} \big ) 
\, =\, (2\alpha) \cdot \frac{2}{\beta} \cdot e^{- 2 \lambda (\xi_{1} - \xi_{2})} \cdot \frac{ \mathfrak c_{1}}{\, \sqrt{2 \, \pi \, t^{3}}\, } \exp \Big \{ - \frac{\, ( \mathfrak c_{1} - \lambda \, t)^{2}}{2 t} \Big \} {\mathrm d} \xi_{1} {\mathrm d} \xi_{2} \, , 
\]
\[
\text{ where } \quad \mathfrak c_{1} \, :=\, \xi_{1} - \Big(\frac{\, 2 + \beta \, }{\beta} \Big) \xi_{2} + x_{1} + \Big( \frac{\, 2 - \beta \, }{\beta} \Big) x_{2} + \frac{\, 2\, }{\beta} \, g \, t \, , 
\]

\medskip
\noindent
if $\,\beta > 0\,$, $\,\xi_{1} \ge \xi_{2}\,$ and $\,\xi_{2} < x_{2} + g\, t \,$. Similarly, by   (skew) symmetry:
\[
\mathbb P \big (X_{1}(t) \in {\mathrm d} \xi_{1}\, , \, \, X_{2}(t) \in {\mathrm d} \xi_{2} \big ) 
\, =\, 2(1-\alpha) \cdot \frac{2}{\beta} \cdot e^{- 2 \lambda (\xi_{2} - \xi_{1})} \cdot \frac{ \mathfrak c_{2}}{\, \sqrt{2 \, \pi \, t^{3}}\, } \exp \Big \{ - \frac{\, ( \mathfrak c_{2} - \lambda \, t)^{2}}{2 t} \Big \} {\mathrm d} \xi_{1} {\mathrm d} \xi_{2} \, , 
\]
\[
\text{ where } \quad \mathfrak c_{2} \,:=\, \xi_{2} - \Big(\frac{\, 2 + \beta \, }{\beta} \Big) \xi_{1} + x_{1} + \Big( \frac{\, 2 - \beta \, }{\beta} \Big) x_{2} + \frac{\, 2\, }{\beta} \, g \, t \, , 
\]

\medskip
\noindent
if $\,\beta > 0\,$, $\,\xi_{2} \ge \xi_{1}\,$ and $\,\xi_{1} < x_{2} + g\, t \,$. If $\,\beta > 0\,$, $\,\xi_{1} > \xi_{2} \, =\, x_{2} + g\,t\,$, then the local time $\,\widehat{L}^{\,Y}(\cdot) \,$ is absent, and  the transition density is easily obtained from (\ref{eq: joint0}): 
\[
\mathbb P \big(X_{1}(t) \in {\mathrm d} \xi_{1}\, , \,X_{2}(t) \, =\, x_{2} + g\, t \big) \, =\, 
\]
\[
=\, \frac{1}{\, \sqrt{2\pi t}\, } \Big( \exp \Big \{ - \frac{\,(a - x_{1} + x_{2} + \lambda t )^{2}\, }{2 t} \Big\} - e^{- 2\lambda \, a } \cdot \exp \Big \{ - \frac{\,(a +  x_{1} - x_{2} + \lambda t )^{2}\, }{2 t} \Big\} \Big)\Big \vert_{a \, =\, \xi_{1} - x_{2} - g \, t} {\mathrm d} \xi_{1}  \, . 
\]

\bigskip
\noindent
$\bullet~$ 
For another extreme example, in the degenerate case with $\, \sigma \, =\, 1\,$, thus $\, \rho \, =\, 0\,$, $\,\gamma \, =\,  -1 \,$ and with $\,x_{1} \ge x_{2}\,$, (\ref{2.23b})-(\ref{2.24}) become 
\[
X_{1}(t) = x_{1} - h \, t - (Y^{-}(t)- y_{0}^{-})+ (2-\beta) \, \widehat{L}^{\,Y}(t) \, , \quad 
X_{2}(t) = x_{2} - h \, t - (Y^{+}(t)-y_{0}^{+}) + (2-\beta) \, \widehat{L}^{\,Y}(t) \, , 
\]
for $\, 0 \le t < \infty \,$. If $\,\beta < 2\,$, $\,\xi_{1} \ge \xi_{2}\,$ and $\,\xi_{1} > x_{1} - h\, t\,$, then 
\[
\mathbb P \big (X_{1}(t) \in {\mathrm d} \xi_{1}\, , \, \, X_{2}(t) \in {\mathrm d} \xi_{2} \big ) 
\, =\, (2\alpha) \cdot \frac{2}{\, 2 - \beta\, } \cdot e^{- 2 \lambda (\xi_{1} - \xi_{2})} \cdot \frac{ \mathfrak c_{3}}{\, \sqrt{2 \, \pi \, t^{3}}\, } \exp \Big \{ - \frac{\, ( \mathfrak c_{3} - \lambda \, t)^{2}}{2 t} \Big \} {\mathrm d} \xi_{1} {\mathrm d} \xi_{2} \, , 
\]
\[
\text{ where } \quad \mathfrak c_{3} \, :=\, \Big(\frac{4 - \beta}{\, 2 - \beta \, } \Big) \xi_{1} - \xi_{2} - \Big(\frac{\beta}{\, 2 - \beta\, } \Big) x_{1} - x_{2} + \Big( \frac{4 - \beta}{\, 2 - \beta \, } \Big) h \, t 
\]
If $\,\beta < 2\,$, $\,\xi_{2} \ge \xi_{1}\,$ and $\,\xi_{2} > x_{1} - h\, t\,$, then 
\[
\mathbb P \big (X_{1}(t) \in {\mathrm d} \xi_{1}\, , \, \, X_{2}(t) \in {\mathrm d} \xi_{2} \big ) 
\, =\, 2( 1- \alpha) \cdot \frac{\,2\, e^{- 2 \lambda (\xi_{2} - \xi_{1})}\, }{\, 2 - \beta\, }    \cdot \frac{ \mathfrak c_{4}}{\, \sqrt{2 \, \pi \, t^{3}}\, } \exp \Big \{ - \frac{\, ( \mathfrak c_{4} - \lambda \, t)^{2}}{2 t} \Big \} {\mathrm d} \xi_{1} {\mathrm d} \xi_{2} \, , 
\]
\[
\text{ where } \quad \mathfrak c_{4} \, :=\, \Big(\frac{4 - \beta}{\, 2 - \beta \, } \Big) \xi_{2} - \xi_{1} - \Big(\frac{\beta}{\, 2 - \beta\, } \Big) x_{1} - x_{2} + \Big( \frac{4 - \beta}{\, 2 - \beta \, } \Big) h \, t \, . 
\]
If $\,\beta < 2\,$, $\,\xi_{1} \, =\, x_{1} - h \, t > \xi_{2}\,$, then the local time $\, \widehat{L}^{\,Y}(\cdot)\, $ does not accumulate, that is, the transitions density is obtained from (\ref{eq: joint0}) : 
\[
\mathbb P (X_{1}(t) \, =\, x_{1} - h \, t \, , \,X_{2}(t) \in {\mathrm d} \xi_{2}) \, =\, 
\]
\[
=\, \frac{1}{\, \sqrt{2\pi t}\, } \Big( \exp \Big \{ - \frac{\,(a - x_{1} + x_{2} + \lambda t )^{2}\, }{2 t} \Big\} - \,e^{- 2\lambda \, a } \, \exp \Big \{ - \frac{\,(a +  x_{1} - x_{2} + \lambda t )^{2}\, }{2 t} \Big\} \Big)\Big \vert_{a \, =\, x_{1} - \xi_{2} - h \, t} {\mathrm d} \xi_{1}  \, . 
\]

\medskip
\noindent
$\bullet~$ 
For the isotropic variance case with $\, \rho \, =\, \sigma \, =\, 1 \, / \, \sqrt{2} \,$, $\,\gamma \, =\, 0\,$ and $\, x_{1} \ge x_{2}\,$ the difference and the sum of $\,X_{1}(\cdot) \,$ and $\,X_{2}(\cdot) \,$ are 
\[
X_{1}(\cdot) - X_{2}(\cdot) \, =\, Y(\cdot) \, , \quad 
X_{1}(\cdot) + X_{2}(\cdot) \, =\, x_{1} + x_{2} + \nu \, t  + 2(2\alpha - 1) \widehat{L}^{\,Y}(\cdot) + Q(\cdot) \, , 
\]
where $\, (Y(\cdot), \widehat{L}^{\,Y}(\cdot))\,$ and $\,Q(\cdot) \,$ are independent. Thus the joint distribution of $\,(X_{1}(\cdot), X_{2}(\cdot)) \,$ are obtained by integrating out the local time. 

\smallskip
If $\,\alpha \in (0, 1) \setminus \{ 1\, / \, 2\}\,$, then the above transition densities are discontinuous on the diagonal line due to the skewness. 
If $\,\alpha = 1\, / \, 2\,$ and $\,\beta \, =\, 1\,$, then these formulae are the same as those of the degenerate system studied in \textsc{Fernholz et al.} \cite{FIKP2012}. 
The transition densities for all the other cases as well as the joint distribution of $\, (X_{1}(\cdot), X_{2}(\cdot), L^{X_{1} - X_{2}}(\cdot)) \,$ are computable from the skew representations (\ref{2.23b})-(\ref{2.24}) and the joint distribution (\ref{eq: joint+})-(\ref{eq: joint0}) in a similar manner.  

\subsection{Time Reversal}
\label{sec7.2} 

\noindent
We consider now the time-reversal 
\begin{equation} \label{eq: 7.1}
\widehat{X}_{i}(t) \, :=\, X_{i}(T-t) \, , \quad \widetilde{X}_{i}(t) \, :=\, X_{i}(T-t) - X_{i}(T) \, , \quad 0 \le t \le T \, ,\, \,  i = 1, 2 \, 
\end{equation}
of the solution to the system (\ref{1.3})-(\ref{1.4}) with the backwards filtration $\, \widetilde{\mathbf F} \, =\, \{\widetilde{\mathfrak F}(t) \}_{0 \le t \le T}\,$ generated by the random variable $\,Y(T)\,$ and by the time-reversal $$\,\big(\, \widetilde{W}(t) \, :=\, W(T-t) - W(T) \, , \,\,\widetilde{Q}(t) \, :=\,  Q(T-t) - Q(T) \, \big)\,, \qquad 0 \le t \le T \,$$ of the planar Brownian motion :  
\[
\widetilde{\mathfrak F}(t) \, :=\, {\bm \sigma} (Y(T)) \vee \mathfrak F^{\,(\widetilde{Q}, \widetilde{W})} (t) \, , \quad \mathfrak F^{\,(\widetilde{Q}, \widetilde{W})}(t) \, :=\, {\bm \sigma} \big(\widetilde{Q}({\bm \theta}), \widetilde{W}({\bm \theta}) \, ; 0 \le {\bm \theta} \le t \big) \, , \quad 0 \le t \le T \, . 
\]

With some extra work in addition to the discussion of Remark \ref{rem: TR} we may show that $\, \widehat{Y}(\cdot)\,$ is a diffusion (\ref{eq: Yh}) driven by the $\,\widehat{\mathbb F}\,$-Brownian motion $\, \widehat{W}^{\sharp}(\cdot) \,$ (c.f. \textsc{Pardoux} \cite{MR942014} 
and section 3 of \textsc{Petit} \cite{MR1464173}).  
Combining the skew representations (\ref{2.23b})-(\ref{2.24}) with the time-reversals (\ref{eq: Yh})-(\ref{eq: Lhh}), we derive the time-reversed skew representation form 

\[
\widetilde{X}_{1}(t) \, =\,  - \mu \, t  + \rho^{2} \big( \widehat{Y}^{+}(t) - \widehat{Y}^{+}(0) \big) - \sigma^{2} \big( \widehat{Y}^{-}(t) - \widehat{Y}^{-}(0) \big) - (1 - \beta - \gamma) \widehat{L}^{Y}(t) + \rho \, \sigma \widehat{Q}(t) \, ~~~~~~
\]
\begin{equation} \label{eq: trX1} 
 ~~~~~~ =\, - \mu \, t  + \int^{t}_{0} \big( \rho^{2} {\bf 1}_{\{ \widehat{Y}(s) > 0\} } + \sigma^{2} {\bf 1}_{\{ \widehat{Y}(s) \le 0\}} \big) {\mathrm d} \widehat{Y}(s) - (1 - \beta - 2\gamma) \widehat{L}^{Y}(t) + \rho \, \sigma \widehat{Q}(t)
\end{equation}

\noindent
for $\,0 \le t \le T\,$,  and 
\[
\widetilde{X}_{2}(t) \, =\,  - \mu \, t  - \sigma^{2} \big( \widehat{Y}^{+}(t) - \widehat{Y}^{+}(0) \big) + \rho^{2} \big( \widehat{Y}^{-}(t) - \widehat{Y}^{-}(0) \big) - (1 - \beta - \gamma) \widehat{L}^{Y}(t) + \rho \, \sigma \widehat{Q}(t) \, ~~~~~~
\]
\begin{equation} 
\label{eq: trX2}
\, ~~~~~~~~~~ =\, - \mu \, t  - \int^{t}_{0} \big( \rho^{2} {\bf 1}_{\{ \widehat{Y}(s) \le 0\} } + \sigma^{2} {\bf 1}_{\{ \widehat{Y}(s) > 0\}} \big) {\mathrm d} \widehat{Y}(s) - (1 - \beta - 2\gamma) \widehat{L}^{Y}(t) + \rho \, \sigma \widehat{Q}(t)\,.
\end{equation}

\begin{rem} By analogy with (\ref{eq: 7.1}) we denote the time-reversal of ranks by $\,\widehat{R}_{i}(t) \, :=\, R_{i}(T-t) \,$ for $\, 0 \le t \le T\,$, $\,i = 1, 2\,$. Applying the \textsc{Tanaka} formula to (\ref{eq: trX1})-(\ref{eq: trX2}), we may derive the time-reversed dynamics of $\, (R_{1}(\cdot), R_{2}(\cdot))\,$. 
\end{rem}

\begin{rem} As we saw in Remarks \ref{rem: STR}-\ref{rem: FS}, the process $\,Y(\cdot)\,$ is strictly time-reversible  when   started at its    invariant distribution (\ref{eq: sta dens}). Under this invariant distribution, the dynamics of the time-reversal of $\,(R_{1}(\cdot), R_{2}(\cdot)) \,$ can be derived through the skew representations of (\ref{2.23b})-(\ref{2.24}). 
\end{rem} 

\bigskip 

\noindent{\bf Acknowledgements}  
The authors are grateful to Drs. Adrian Banner, Vassilios Papathanakos, Phillip Whitman and Mykhaylo Shkolnikov for   several helpful discussions, and to Dr. Vilmos Prokaj for his very careful reading of the manuscript and his many suggestions. The research of the third author was supported in part by National Science Foundation Grant DMS-09-05754.


\medskip 


\makeatletter
\def\barxiv#1{{\url@fmt{ArXiv e-print:
    }{\upshape\ttfamily}{#1}{\arxiv@base#1}}}
\def\bdoi#1{{\url@fmt{DOI:
    }{\upshape\ttfamily}{#1}{\doi@base#1}}}
\def\bdoi#1{\@ifnextchar.\@gobble\relax}
\makeatother


\begin{thebibliography}{30}

\medskip
{\small 
\bibitem[\protect\citeauthoryear{\textsc{Anulova} (1980)}{1980}]{MR609190} {S.V. Anulova},   Diffusion processes with singular characteristics. In ``Stochastic Differential Systems, Filtering and Control'':  Proceedings of an  I.F.I.P.-W.G. Conference, Vilnius, 1978, Lithuania. {\it  Lecture Notes in Control and Information Systems} {\bf 25} (1980) 264-269. Springer-Verlag,  New York.   

\bibitem[\protect\citeauthoryear{\textsc{Appuhamillage, Th.,  Vrushali, B., Thomann, E., Waymire, E.  \& Wood, B} (2011)}{2011}]{MR2759199} {Th. Appuhamillage,  B. Vrushali, E. Thomann,  E. Waymire, B. Wood}, Occupation and local times for skew Brownian motion with application to dispersion across an interface.  {\it Ann. Appl. Probab.} {\bf 21} (2011) 183-214. (Correction: ibid., 2050-2051.)

\bibitem[\protect\citeauthoryear{\textsc{R. Bass, \& Chen, Z.Q.} (2005)}{2005}]{MR2203887} {R. Bass, Z.Q. Chen},  One-dimensional stochastic differential equations with singular and degenerate co\"efficients. {\it  Sankhy$\bar{a}$}   {\bf 67} (2005) 19-45.

\bibitem[\protect\citeauthoryear{\textsc{Burdzy, K. \& Nualart, D.}  (2002)}{2002}]{MR1902187} {K. Burdzy, D. Nualart,}  Brownian motion  reflected on  Brownian motion. {\it  Probab. Theory Relat. Fields}  {\bf 122}  (2002) 471-493.

\bibitem[\protect\citeauthoryear{\textsc{Chitashvili, R.J. \& Lazrieva,  N.L.} (1981)}{1981}]{MR636083} {R.J. Chitashvili, N.L. Lazrieva},    Strong solutions of stochastic differential equations with boundary conditions. {\it  Stochastics} {\bf 5} (1981) 255-309.

\bibitem[\protect\citeauthoryear{\textsc{Engelbert, H.J. \& Schmidt, W.} (1984)}{1984}]{MR798317} {H.J. Engelbert, W. Schmidt}, 
On one-dimensional stochastic differential equations with generalized drift. {\it Lecture Notes in Control and Information Systems}   {\bf   69} (1985) 143-155. Springer-Verlag, NY. 

\bibitem[\protect\citeauthoryear{\textsc{Fernholz, E.R., Ichiba, T.,    Karatzas, I. \& Prokaj, V.}   (2012) }{2012}]{FIKP2012} {E.R. Fernholz, T. Ichiba, I. Karatzas, V. Prokaj},    A planar diffusion with rank-based characteristics, and perturbed Tanaka equations. {\it Probab. Theory Relat. Fields,} to appear.  Available at the site {\it http://arxiv.org/pdf/1108.3992} (posted August 19, 2011).

\bibitem[\protect\citeauthoryear{\textsc{Fukushima, M. \& Stroock, D.} (1986)}{1986}]{MR875449}  {M. Fukushima, D. Stroock}, Reversibility of solutions to martingale problems. In ``Probability, Statistical Mechanics, and Number Theory''. {\it Adv. Math. Suppl. Stud.} {\bf 9} (1986) 107-123. Academic Press, Orlando, FL. 



\bibitem[\protect\citeauthoryear{\textsc{Harrison, J.M. \& Shepp, L.A.} (1981) }{1981}]{MR606993} {J.M. Harrison, L.A. Shepp},   On skew Brownian motion. {\it Ann. Probab.} {\bf  9} (1981)  309--313. 



\bibitem[\protect\citeauthoryear{\textsc{It\^o, K. \& Mc$\,$Kean, H.P., Jr.}  (1963)}{1963}]{MR0154338} {K. It\^o, H.P. Mc$\,$Kean Jr.},    Brownian Motion on a Half-Line. {\it Illinois J. Math.} {\bf  7} (1963) 181--231.


\bibitem[\protect\citeauthoryear{\textsc{It\^o, K. \& Mc$\,$Kean, H.P., Jr.}  (1974)}{1974}]{MR0345224} {K. It\^o, H.P. Mc$\,$Kean Jr.},  {\it Diffusion Processes and Their Sample Paths.}  Second Printing (Corrected), Springer Verlag,   New York, 1974.

\bibitem[\protect\citeauthoryear{\textsc{Karatzas, I. \& Shreve, S.E.}  (1984) }{1984}]{MR744236} {I. Karatzas, S.E. Shreve},   
Trivariate density of Brownian motion, its local and occupation times, with application to stochastic control.  {\it Ann. Probab.} {\bf 12} (1984) 819-828.

\bibitem[\protect\citeauthoryear{\textsc{Karatzas, I. \& Shreve, S.E.}  (1991)}{1991}]{MR1121940} {I. Karatzas, S.E. Shreve,}   
{\it Brownian Motion and Stochastic Calculus.} 
Second Edition, Springer Verlag, New York, 1991.

\bibitem[\protect\citeauthoryear{ \textsc{Lang, R.}  (1995)}{1995}]{MR1340556} 
 {R. Lang},  Effective conductivity and skew Brownian motion. {\it  J. Stat. Phys.} {\bf 80} (1995) 125-146.





\bibitem[\protect\citeauthoryear{\textsc{Le$\,$Gall, J.F.}  (1983)}{1983}]{MR0770393} {J.F. Le$\,$Gall, }   
Applications des temps locaux aux \'equations diff\'erentielles
stochastiques unidimensionelles. 
{\it Lecture Notes in Mathematics}   {\bf  986}  (1983) 15-31. Springer-Verlag, New York.

\bibitem[\protect\citeauthoryear{\textsc{Le$\,$Gall, J.F.}  (1984) }{1984}]{MR777514} {J.F. Le$\,$Gall, }   
One-dimensional stochastic differential equations
involving the local times of the unknown process. {\it Lecture Notes in Mathematics}   {\bf  1095}   (1984) 51-82. Springer-Verlag, New York.




\bibitem[\protect\citeauthoryear{\textsc{Lejay, A.}  (2006) }{2006}]{MR2280299} {A. Lejay},  
On the constructions of the skew Brownian motion. 
{\it Probability Surveys} {\bf 3} (2006) 413-466.




\bibitem[\protect\citeauthoryear{\textsc{Nakao, S.}  (1972) }{1972}]{MR0326840} {S. Nakao},  
On the pathwise uniqueness of solutions of one-dimensional stochastic differential equations.  {\it Osaka J. of Math.} {\bf 9} (1972) 513-518.



\bibitem[\protect\citeauthoryear{\textsc{Ouknine, Y. \& Rutkowski, M.} (1995)}{1995}]{MR1324099} {Y. Ouknine, M. Rutkowski,}  Local times of functions of continuous
semimartingales {\it Stochastic Anal. Appl.} {\bf 13} (1995) 211-231. 

\bibitem[\protect\citeauthoryear{\textsc{Pardoux, E.} (1986)}{1986}]{MR942014} {E. Pardoux},   Grossissement d'une filtration et retournement du temps. {\it S\'eminaire de Probabilit\'es} XX, 1984/85, {\it Lecture Notes in Mathematics}, {\bf 1204} (1986) 48-55. Springer-Verlag, New York. 

\bibitem[\protect\citeauthoryear{\textsc{Petit, F.}  (1997)}{1997}]{MR1464173} {F. Petit},    Time reversal and reflected diffusion. {\it  Stochastic Process. Appl.} {\bf 69} (1997) 25-53.

\bibitem[\protect\citeauthoryear{\textsc{Portenko, N.I.}  (1976) }{1976}]{MR0440716} {N.I. Portenko},   
 Generalized diffusion processes.    {\it Lecture Notes in Mathematics}   {\bf  550} (1976) 500-523. Springer-Verlag, New York.
 
\bibitem[\protect\citeauthoryear{\textsc{Portenko, N.I.} (1979)}{1979}]{MR522237} {N.I. Portenko, } 
Diffusion processes with generalized drift co\"efficients.  {\it Theory Probab. Appl.} {\bf 24} (1979) 62-78. 

\bibitem[\protect\citeauthoryear{\textsc{Portenko, N.I.}  (1979.b)}{1979}]{MR532446}
{N.I. Portenko},   
Stochastic differential equations  with generalized drift vector.  {\it Theory Probab. Appl.} {\bf 24}  (1979) 338-353.

\bibitem[\protect\citeauthoryear{\textsc{Portenko, N.I.}  (1990)}{1990}]{MR1104660} {N.I. Portenko,}  
{\it Generalized Diffusion Processes}.    Translations of Mathematical Monographs. American Mathematical Society, Providence, RI  1990.

\bibitem[\protect\citeauthoryear{\textsc{Prokaj} (2010)}{2011}]{tanaka2009}
{V. Prokaj,}  
The solution  of the perturbed Tanaka equation is pathwise unique.  {\it Ann. Probab.}, to appear. 
 (2011) arXiv:1104.0740. 

\bibitem[\protect\citeauthoryear{\textsc{Soucaliuc, F., T\'oth, B. \& Werner, W.  } (2000)   }{2000}]{MR1785393} 
{F. Soucaliuc, B. T\'oth, W. Werner, } 
Reflection and coalescence between independent one-dimensional Brownian
motions.
{\it Ann. Inst. Henri Poincar\'e, Sec. B} {\bf 36}  (2000)  509-545.

\bibitem[\protect\citeauthoryear{\textsc{Soucaliuc, F.  \& Werner, W.  } (2002)    
}{2002}]{MR1917545}
{F. Soucaliuc, W. Werner,}     
A note on reflecting  {B}rownian motions. 
{\it Electron. Comm. Probab.} {\bf 7} (2002) 117-122. 

\bibitem[\protect\citeauthoryear{\textsc{Sznitman, A.S. \& Varadhan, S.R.S.  }  }{1986}]{MR833269}
{A.S. Sznitman, S.R.S. Varadhan, }    A multidimensional process involving local time.   {\it Probab. Theory Relat. Fields} {\bf 71} (1986)  553-579. 

\bibitem[\protect\citeauthoryear{  {Walsh, J.B.} (1978) 
}{1978}]{TempsLocaux78} 
  {J.B. Walsh, } 
  A diffusion with   discontinuous local time. 
  In ``Temps Locaux", {\it Ast\'erisque} {\bf 52-53}  (1978) 37-45.

}
\end{thebibliography}
  \end{document}